\documentclass[a4paper]{amsart}

\usepackage[utf8]{inputenc}
\usepackage[english]{babel}
\usepackage{color}
\usepackage{hyperref}
\usepackage{float}
\usepackage{tikz}
\usepackage[normalem]{ulem} 
\usetikzlibrary{automata}
\usepackage{mathtools}
\usepackage[margin=3cm]{geometry}

\usepackage{amsthm}
 \newtheorem{theorem}{Theorem}[section]
 \newtheorem{lemma}[theorem]{Lemma}
 \newtheorem{corollary}[theorem]{Corollary}
 \newtheorem{proposition}[theorem]{Proposition}

  \newtheorem{introthm}{Theorem}
  \newtheorem{introcor}[introthm]{Corollary}

\theoremstyle{definition}
\newtheorem{definition}[theorem]{Definition}
\newtheorem{claim}[theorem]{Claim}
\newtheorem{example}[theorem]{Example}
\newtheorem{remark}[theorem]{Remark}

\newtheorem{problem}{Problem}

\usepackage{amssymb}

\newcommand{\Cone}{\text{Cone}}

\tikzset{vertex/.style={circle, draw, fill=black!50},inner sep=0pt, minimum width=4pt}

\newtheorem{convention}[theorem]{Convention}
\newcommand{\cay}{\operatorname{Cay}}
\newcommand{\mc}{\mathcal}

\DeclareMathOperator{\CAT}{CAT}

\definecolor{Green}{RGB}{30, 150, 60}

\title{Regularity of Morse geodesics and growth of stable subgroups}
\author{Matthew Cordes}
\address{ETH Zurich, Raemistrasse 101, 8092, Zurich, Switzerland}
\email{matthew.cordes@math.ethz.ch}

\author{Jacob Russell}
\address{Rice University, 6100 Main St, Houston, TX 77005, USA}
\email{jacob.russell@rice.edu}

\author{Davide Spriano}
\address{ETH Zurich, Raemistrasse 101, 8092, Zurich, Switzerland}
\email{davide.spriano@math.ethz.ch}

\author{Abdul Zalloum}
\address{Queen's University, 99 University Ave, Kingston, ON K7L 3N6, Canada}
\email{az32@queensu.ca}

\begin{document}
\maketitle
\begin{abstract}
We prove that Morse local-to-global groups grow exponentially faster than their infinite index stable subgroups. This generalizes a result of Dahmani, Futer, and Wise in the context of quasi-convex subgroups of hyperbolic groups to a broad class of groups that contains the mapping class group, CAT(0) groups, and the fundamental groups of closed 3-manifolds. To accomplish this, we develop a theory of automatic structures on Morse geodesics in Morse local-to-global groups. Other applications of these automatic structures include a description of stable subgroups in terms of regular languages,  rationality of the growth of stable subgroups, density in the Morse boundary of the attracting fixed points of Morse elements, and containment of the Morse boundary inside the limit set of any infinite normal subgroup.
\end{abstract}

\section{Introduction}

Growth and curvature are central themes in geometric group theory. These ideas are frequently intertwined, with a particularly expansive literature existing around  growth in groups with features of negative curvature. In the present work, we study the growth of geometrically well-behaved groups in a large class of groups exhibiting hyperbolic-like behavior. 

 If $G$ is a group with finite generating set $A$, then the \emph{growth function} of a subgroup $H$ with respect to $A$ is the function $f_{H,A} \colon \mathbb{N} \to \mathbb{N}$ where $f_{H,A}(n)$ is the number of elements of $H$ whose distance from the identity in the Cayley graph of $G$ with respect to $A$ is at most $n$. The \emph{growth rate} of the subgroup $H$ with respect to $A$  is denoted $\lambda_{H,A}$ and defined as \[\lambda_{H,A} = \limsup_{n\to \infty} \sqrt[n]{f_{H,A}(n)}.\] Both the growth function and the growth rate of $H$ are highly dependent on the choice of generating set for $G$.

Dahmani, Futer, and Wise  recently proved that  quasi-convex subgroups of hyperbolic groups exhibit a \emph{growth rate gap} \cite{DahmaniFuterWise2018}.  That is, if $H$ is an infinite index quasi-convex subgroup of a hyperbolic group $G$, then the growth rate of $H$ is strictly less than the growth rate of $G$, regardless of the choice of finite generating set. One of the concluding results of the present work extends this result to a variety of non-hyperbolic groups that contain infinite \emph{Morse geodesics}, a hallmark of hyperbolic-like behavior. This is accomplished by developing a theory of regular languages for Morse geodesics in these groups that is analogous to the celebrated   automatic structures for hyperbolic groups first studied by Cannon.

\subsection{Statement of Results.}  Given a function $M \colon [1,\infty) \times [0,\infty) \to [0,\infty)$, a (quasi-) geodesic $\gamma$ is  $M$-Morse if  every $(k,c)$-quasi-geodesic with endpoints on $\gamma$ is contained in the $M(k,c)$-neighborhood of the subsegment of $\gamma$ between its endpoints. The function $M$ is called a \emph{Morse gauge}. Because a group is hyperbolic if and only if every infinite geodesic in the Cayley graph is Morse, Morse geodesics are a sign of hyperbolic-like behavior \cite{Cashen_Mckay_boundary,Gromov}. A large number of important, non-hyperbolic groups contain infinite Morse geodesics \cite{Behrstock_pA_are_Morse, Sisto_Morse_in_aH,OOS_Lacunary_hyperbolic_groups}, allowing fruitful generalizations of  results from the theory of hyperbolic groups \cite{Cordes2017,DMS_divergence,ChSu2014,AMST_intersection_of_stable_subgroups}. In the context of Morse geodesics, the analogous notion to quasi-convex subgroups of hyperbolic groups are the \emph{stable subgroups} introduced by Durham and Taylor \cite{DurhamTaylor2015}. Stable subgroups are precisely the subgroups that are quasi-convex with respect to Morse geodesics with a fixed Morse gauge; see Definition \ref{def:stable_subgroup}. We prove that Dahmani, Futer, and Wise's growth rate gap persists for stable subgroups of groups that satisfy a local-to-global property for their Morse quasi-geodesics.

\begin{introthm}[Growth rate gap for stable subgroups]\label{intro_thm:growth_torsion_free}
	Let $G$ be a finitely generated and virtually torsion-free group with the Morse local-to-global property. For any finite generating set $A$ of $G$ and any infinite index stable subgroup $H <G$, the growth rate of $H$ with respect to $A$ is strictly smaller than the growth rate of $G$ with respect to $A$. That is, \[\lambda_{H,A} < \lambda_{G,A}.\]
	In particular, this conclusion holds when $G$ is virtually either the mapping class group of an orientable finite type surface, the fundamental group of a compact special cube complex, or the fundamental group of a closed $3$-manifold.
\end{introthm}

A finitely generated group has the \emph{Morse local-to-global} property if every path in the Cayley graph that is a Morse quasi-geodesic at a sufficiently large local scale is globally a Morse quasi-geodesic; see Definition \ref{defn:Morsel_local_to_global}  for the precise definition.
Gromov showed that hyperbolic spaces are characterized by local quasi-geodesics being global quasi-geodesics \cite{Gromov}; the Morse local-to-global property is  a generalization of this phenomena.
Morse local-to-global groups were introduced by Russell, Spriano, and Tran who showed they include the mapping class group of an orientable, finite type surface, all cocompact $\CAT(0)$ groups, the fundamental group of any closed $3$-manifold, and any group hyperbolic relative to Morse local-to-global groups \cite{Morse-local-to-global}. We will describe below the two places in the proof of Theorem \ref{intro_thm:growth_torsion_free} where the Morse local-to-global property is used.  In both places, we give examples that demonstrate our techniques would not apply to the broader setting of all finitely generated groups that contain infinite Morse geodesics.

If we remove the requirement that $G$ is virtually torsion-free, we can still recover the growth rate gap when the stable subgroup is residually finite. Since stable subgroups are always hyperbolic groups, there  are no known examples of stable subgroups that are not residually finite. 

\begin{introthm}[Growth rate gap for residually finite stable subgroups]\label{intro_thm:growth_res_finite}
	Let $G$ be a Morse local-to-global group and let $H$ be an infinite index stable subgroup of $G$. If $H$ is residually finite, then for any finite generating set $A$ of $G$, the growth rate of $H$ with  respect to $A$ is strictly smaller than the growth rate of $G$ with respect to $A$.
\end{introthm}

In several specific classes of Morse local-to-global groups, stable subgroups correspond with independently studied subgroups. For example, Koberda, Mangahas, and Taylor prove the stable subgroups of right-angled Artin groups are the purely loxodromic subgroups \cite{KoberdaMangahasTaylor2017}, while  Durham and Taylor showed that the stable subgroups of the mapping class group are precisely the convex cocompact subgroups of Farb and Mosher \cite{DurhamTaylor2015}. Because the mapping class group is virtually torsion-free \cite{primer},  we can therefore apply Theorem \ref{intro_thm:growth_torsion_free} to produce a definite growth rate gap for convex cocompact subgroups of the mapping class group.

\begin{introcor}[Growth rate gap for convex cocompact subgroups]\label{intro_cor:MCG_case_of_Growth}
	Let $G$ be the mapping class group of an orientable, finite type surface and let $H$ be a convex cocompact subgroup of $G$. For any finite generating set $A$ of $G$, the growth rate of $H$ with respect to $A$ is strictly less than the growth rate of $G$ with respect to $A$.
\end{introcor}

Corollary \ref{intro_cor:MCG_case_of_Growth} is a companion result---achieved by very different means---to a theorem of Gekhtman \cite{Gekhtmen_convex_cocompact}.  Gekhtman establishes  a growth rate gap between the orbit of a convex cocompact subgroup and the entire mapping class group in Teichm\"uller space. Corollary \ref{intro_cor:MCG_case_of_Growth} proves the same growth rate gap  appears  in the Cayley graph of the mapping class group with respect to any finite generating set.

Dahmani, Futer, and Wise's proof of the growth rate gap for quasi-convex subgroups of hyperbolic groups rests on a  seminal result of Cannon that the geodesic words in a hyperbolic group form a regular language, regardless of choice of generating set.  Accordingly, the key  step in establishing the gap for Morse local-to-global groups is to prove that for each Morse gauge $M$, the language of $M$-Morse geodesic words in the group is contained in a regular language whose elements are all uniformly $M'$-Morse.

\begin{introthm}[Regular languages containing Morse geodesics]\label{intro_thm:regular_language_for_Morse_Geodesics}
	Let $G$ be a finitely generated group with the Morse local-to-global property and let $A$ be a finite generating set. For each Morse gauge $M$, there exists a regular language $L_M$ such that	
	\begin{enumerate}
		\item every $M$-Morse geodesic word of $G$ is contained in $L_M$;
		\item every element of $L_M$ is an $M'$-Morse geodesic word where $M'$ is determined by $M$.		
	\end{enumerate}
	Further, for any choice of total order on the generating set $A$, the short-lex sublanguage of $L_M$ is also regular.
\end{introthm}
Theorem \ref{intro_thm:regular_language_for_Morse_Geodesics} provides an extension of a result by Eike and Zalloum  who showed that strongly contracting geodesics with a fixed parameter form a regular language in any finitely generated group \cite{EikeZalloum}. Strongly contracting  geodesics are a stronger form of hyperbolic-like behavior in a group compared to Morse geodesics. In particular, being contracting is intrinsically a local property,  while Morse geodesics are not in general determined locally. A more detailed comparison of Morse and contracting geodesics in the context of this paper is given in Subsection \ref{subsec:contracting}.

The bridge from Theorem \ref{intro_thm:regular_language_for_Morse_Geodesics} to our results on growth is the following  description of stable subgroups of Morse local-to-global groups using regular languages. This is a Morse geodesic version of the Gersten and Short result that a subgroup $H$ of a hyperbolic group is quasi-convex if and only if the language of geodesic words representing $H$ is regular \cite{GerSho1991}.

\begin{introthm}[Regular language description of stability]\label{intro_thm:stable_subgroups}
	Let $G$ be a Morse local-to-global group generated by the finite set $A$. For a subgroup $H < G$, let $L_H$ be the language of geodesic words representing elements of $H$. The subgroup $H$ is stable if and only if $L_H$ is a regular language all of whose elements are $M$-Morse for some Morse gauge $M$. Moreover, if $H$ is stable, there exists a sublanguage $J_H \subseteq L_H$ so that $J_H$ is regular and each element of $H$ is represented by exactly one word in $J_H$.
\end{introthm}

Theorems \ref{intro_thm:regular_language_for_Morse_Geodesics} and \ref{intro_thm:stable_subgroups} make up the core machinery this paper and lay a foundation to extend other results from the theory of automatic groups. In Subsection \ref{subsec:problems}, we discuss a number of open problems in this direction.

Applying a classic result on the rationality of regular languages to the language $J_H$ in Theorem \ref{intro_thm:stable_subgroups}, we have that stable subgroups of Morse local-to-global groups have rational growth function. For the mapping class group, this answers a question of Farb in the case of convex cocompact subgroups \cite[Chapter 2, Question 3.14]{FarbProblems}.

\begin{introcor}[Stable subgroups have rational growth]
Let $H$ be a stable subgroup of a Morse local-to-global group $G$. For any finite generating set $A$ of $G$, there exist polynomials $P(x),Q(x) \in \mathbb{Q}[x]$ such that \[\sum_{n=0}^\infty f_{H,A}(n) x^n = \frac{P(x)}{Q(x)}.\]
\end{introcor}

As an example of the possible broader application of having regular languages for Morse geodesics, we prove a pair of results about the dynamics of  the action of a Morse local-to-global group  on its Morse boundary---the Morse geodesic analogue of the Gromov boundary of a hyperbolic group. We show that the set of attracting fixed points of elements generating cyclic stable subgroups, a.k.a \emph{Morse elements}, are dense in the Morse boundary and that the limit set of an infinite normal subgroup is the entire Morse boundary of the ambient group. Both of these results extend long established results for the Gromov boundary of hyperbolic groups.

\begin{introthm}[Dynamics on the Morse boundary]\label{intro_thm:boundary_applications}
Let $G$ be a Morse local-to-global group and let $\partial_{*} G$ denote its Morse boundary.
\begin{enumerate}
	\item The set of attracting fixed points of Morse elements of $G$ is dense in $\partial_{*}G$.
	\item If $H$ is an infinite normal subgroup of $G$, then  the limit set of $H$ in $\partial_{*} G$ is all of $\partial_{*} G$.
\end{enumerate}
\end{introthm}

A simple, yet interesting, corollary of Theorem \ref{intro_thm:boundary_applications} is that if  the normal subgroup $H$ is a hyperbolic group, then there cannot exist a Cannon--Thurston map from $\partial_{*} H$ to $\partial_{*} G$; see Corollary \ref{cor:no_cannon-thurston}. In particular, the Birman exact sequence does \emph{not} produce a Cannon--Thurston map from the boundary of the fundamental groups of a surface to the Morse boundary of the mapping class group of that surface with an additional puncture.
This result gives a stark difference between the Morse boundary and the Gromov boundary were the existence of such maps was established by Mj \cite{Mitra1998}.

\subsection{Outline of the proofs}\label{subsec:outline} The proof of our main  result, the containment of $M$-Morse geodesic words in a regular language all of whose elements are $M'$-Morse (Theorem \ref{intro_thm:regular_language_for_Morse_Geodesics}), aims to adopt Cannon's proof that the geodesic words in a hyperbolic group form a regular language to the setting of Morse geodesics \cite{Cannon1984}. The key step in Cannon's proof is proving that a hyperbolic group has a finite number of \emph{cone types}. The first approximation to this approach using Morse geodesics would be to define an \emph{$M$-Morse cone type} that records the ways of continuing an $M$-Morse geodesic word to produce an $M$-Morse geodesic word. However, despite the presence of the Morse local-to-global property, Cannon's proof of finitely many cone types fails in this case. The primary difficulty is that the Morse gauge $M$ may increase when applying the Morse local-to-global property. That is, it is only possible to ensure a geodesic that is locally $M$-Morse is globally $M'$-Morse for some $M'$ depending on $M$. 

To circumvent this problem, we define the $M$-Morse cones to keep track of the Morse gauge of the geodesics at a local instead of global scale (Definition \ref{defn:Morse_cones}). We then rely on the Morse local-to-global property to ensure that the locally $M$-Morse geodesics possess enough features of hyperbolicity to  allow for Cannon's proof to be adopted nearly verbatim (Theorem \ref{thm: finite many cone types}).  The cost of this choice is that the obtained regular language $L_M$ contains more words than the na{\"i}ve approach would suggest---$L_M$ is the set of the words that are locally $M$-Morse, but not necessarily globally $M$-Morse. However, since the Morse local-to-global property  ensure that there is uniform Morse gauge $M'$ so that each element of $L_M$ is globally $M'$-Morse, this  does not end up producing any difficulties in the applications.

The regularity of $L_M$ is the first essential place where the Morse local-to-global property is used in our work. If there is a regular language $L$ containing all of the $M$-Morse geodesic words in a group $G$ so that every element of $L$ is uniformly Morse, then $G$  must contain an infinite order (Morse) element; see Remark \ref{rem:infinite_order_element}. Examples of Fink  show that there exist finitely generated, infinite, torsion groups that have infinite Morse geodesics \cite{Fink}, and hence  there is no hope for Theorem \ref{intro_thm:regular_language_for_Morse_Geodesics}  to hold in the context of any group with infinite Morse geodesics.

In our description of stable subgroups (Theorem \ref{intro_thm:stable_subgroups}),  the backward direction does not require the Morse local-to-global property as it is  a special case of the general observation that if a language of words representing elements of a subgroup is regular, then that subgroup is quasi-convex with respect to the paths  determined by the words in the language (Lemma \ref{lem:regular_languages_of_representative_implies_quasi-convex}). For the forward direction, we show that for any group, the language of all words that represent elements of the the subgroup $H$ and whose paths stay within a fixed distance from $H$ is regular (Lemma \ref{lem:quasiconvex_subgroups_regular_language}). When $H$ is stable and $G$ has the Morse local-to-global property, intersecting this language with the regular language $L_M$ for the right Morse gauge $M$ will produce the language of geodesic words representing element of $H$.

Our growth rate results (Theorems \ref{intro_thm:growth_torsion_free} and \ref{intro_thm:growth_res_finite}) follow a commonly used approach.  Given an infinite index stable subgroup $H$, we use either the virtual torsion-freeness of $G$ or the residually finiteness of $H$ to produce an element $g \in G$ and a finite index subgroups $H' < H$ so that $\langle H',g\rangle \cong H' \ast \langle g \rangle$. In both cases, this relies on combination theorems for Morse local-to-global groups proved by Russell, Spriano, and Tran \cite{Morse-local-to-global}. Since finite index subgroups of stable subgroups are stable, we can apply Theorem \ref{intro_thm:stable_subgroups} to $H'$ to produce a regular language $J_{H'}$ of geodesic words that uniquely represent the elements of $H'$. Since $\langle H',g\rangle \cong H' \ast \langle g \rangle$, we can copy a construction of Dahmani, Futer, and Wise to extend $J_{H'}$ to a regular language $J^{+g}$ for $\langle H',g\rangle$. Continuing to follow  Dahmani, Futer, and Wise, we apply some basic Perron--Frobenius theory to prove that the growth rate of words in $J_{H'}$ is strictly less than the growth rate for words in $J^{+g}$. Since these languages biject onto the subgroups $H'$ and  $\langle H',g\rangle$, the growth rates of these languages are the same as the growth rates of the subgroups, producing the desired inequality.

The aforementioned combination theorems for stable subgroups is the second essential place we use the Morse local-to-global property. Osin, Ol'shanskii, and Sapir  have produced an example of a finitely generated group that contains many infinite Morse geodesics,  but cannot satisfy  the combinations theorems as every proper subgroup is cyclic \cite{OOS_Lacunary_hyperbolic_groups}. These combination theorems also introduce our need for the additional hypotheses on either $G$ or $H$. When $G$ is hyperbolic, Dahmani, Futer, and Wise use ping-pong on the Gromov boundary to produce a combination theorem for all hyperbolic groups that allows them to run a modified version of the above argument. The na\"ive approach of extending their proof of this combination theorem cannot be generalized to Morse local-to-global groups because the Morse boundary fails to be compact if the group is not hyperbolic \cite{Murray2015TopologyAD,CordesDurham2017}.

\subsection{A comparison with growth tightness and contracting geodesics}\label{subsec:contracting}

For hyperbolic groups, there are different approaches to prove the growth rate gap for quasi-convex subgroups. In particular, one can use growth tightness techniques that 
rely on strongly contracting geodesics \cite{ACT_growth_tightness, DahmaniFuterWise2018}. As mentioned above, strongly contracting geodesics are another notion of geodesics with hyperbolic-like behavior. 
Although  contracting geodesics coincide with Morse geodesics in hyperbolic and CAT(0) spaces, in general the former condition is strictly stronger.
Notably, Morse geodesics
are preserved under quasi-isometry whereas contracting geodesics are not.
This constitutes one  the main advantages of considering Morse quasi-geodesics, as our results are automatically independent of the choice of generating set. 

Morse geodesics are also more widespread  than contracting geodesics. For instance the mapping 
class group is known to be a Morse local-to-global group, but is not known whether, given a specific generating set, its Cayley graph contains any contracting geodesics. In fact, Rafi and Verberne have shown that, for a certain choice of generating set, there are Morse geodesics in the mapping class group of a 5-punctured sphere that are not strongly
 contracting \cite{Rafi_Verberne_contracting}. 

\subsection{Open Problems}\label{subsec:problems}

\begin{problem}\label{prob:language_and_hyp_groups}
 Gersten and Short proved that if $H$ is a subgroup of a hyperbolic group, then $H$
 is quasi-convex if and only if the language $L_H$ of all geodesic words that start at $e$ and end in $H$ is regular \cite[Theorem 3.1]{GerSho1991}. This theorem exhibits an intrinsic connection between quasi-convexity properties and regular languages  in hyperbolic groups, and inspires the question of whether or not stable subgroups of Morse local-to-global groups can similarly be characterized by regular languages without any external appeal to a Morse gauge. We propose the following conjecture  to resolve this question: For a Morse local-to-global group $G$ with finite generating set $A$, let $H$ be a subgroup of $G$ and $L_H^{(k, c)}$ be the language of $(k,c)$-quasi-geodesic words that start at $e$ and end in $H$.  We conjecture that

\begin{center}
 $H$ is stable $\iff$    $L_H^{(k, c)}$ is regular whenever $k$ is rational.
\end{center}

 This conjecture is inspired by a result of Holt and Rees that showed for hyperbolic groups, the language $L^{(k,c)}$ of all $(k,c)$-quasi-geodesic words is regular if and only if $k$ is rational \cite{HoltRees2003}. Since stable subgroups can be characterized as hyperbolic subgroups with a strong quasi-convexity property (see Theorem \ref{thm:stable iff hyperbolic and Morse}), the key step to resolving our conjecture for stable subgroups appears to be proving a converse to Holt and Rees' result. That is, prove that  for a finitely generated group, if for all rational $k$, the language $L^{(k,c)}$ is regular for some finite generating set, then $G$ is hyperbolic. This result for hyperbolic groups appears to be open and quite interesting.
\end{problem}

\begin{problem}
Every regular language is recognized by a finite directed labeled graph.
This suggests that an infinite regular language enjoys some local-to-global properties as it is an infinite object that is produced by a finite object (a finite graph). Therefore, it  is possible that there is a converse to Theorem \ref{intro_thm:regular_language_for_Morse_Geodesics}. For example, one may hope that if the language of $M$-Morse geodesic words in a group $G$ is regular for each Morse gauge $M$, then $G$ has the Morse local-to-global property.
A more fine tuned converse is also possible as the assumption that the language of $M$-Morse geodesic word is regular is strictly stronger than our conclusion of Theorem \ref{intro_thm:regular_language_for_Morse_Geodesics}. In line with the proposed result for hyperbolic groups in Problem \ref{prob:language_and_hyp_groups}, it is also possible that in order to conclude a group is Morse local-to-global, one might actually need the regularity of a  family of  languages for each Morse gauge instead of a single language.
\end{problem}

\begin{problem}
    Intuitively, a regular language over a finite set $A$ can be thought of as a finite uniform algorithm that takes a finite word $w$ in $A$ and decides in a finite time whether or not the word $w$ is in the regular language. In the case of automatic groups, the membership problem for quasi-convex subgroups has a quadratic time solution \cite{Word_Processing,GerSho1991}. Therefore, we pose two questions: Is the membership problem for stable subgroup of Morse local-to-global groups solvable? If so, does Theorem \ref{intro_thm:stable_subgroups} have any implications for the complexity of the membership problem of stable subgroups in Morse local-to-global groups?
\end{problem}

\begin{problem}
Using automatic structures, I. Kapovich produced an  algorithm  that inputs a finite set  $S$ of elements of a hyperbolic group $G$ and halts if  $S$ generates
a quasi-convex subgroup of $G$, but  runs forever otherwise \cite{Ilya}. H. Kim gives various detection
and decidability algorithms for stability of a finitely generated subgroup of mapping
class groups, right-angled Artin groups, and toral relatively hyperbolic groups \cite{Kim2020}.  Possibly with the aid of Theorem \ref{intro_thm:regular_language_for_Morse_Geodesics}, can one  produce detectability algorithms for the stable subgroups of  Morse local-to-global groups? 
\end{problem}

\begin{problem}
Dahmani, Futer, and Wise were able to prove that the growth rate of an infinite index quasi-convex subgroups is strictly smaller than that of the ambient hyperbolic group  with no additional assumptions \cite{DahmaniFuterWise2018}.  Accordingly, we ask if the  torsion-free or residually finite assumptions in our growth rate results can be removed. We also wonder if one could prove that the growth rate of an infinite index stable subgroup is strictly smaller than the growth rate of the ambient group without the Morse local-to-global property. If such a result is possible, it would require a different approach than we take here as the  automatic structures for Morse geodesic words we build in this paper cannot exist in a general finitely generated group; see Remark \ref{rem:infinite_order_element}.
\end{problem}

\begin{problem}
In the setting of cocompactly cubulated groups, Dahmani, Futer, and Wise also establish a growth rate gap for subgroups that stabilize an essential hyperplane \cite[Theorem 1.3]{DahmaniFuterWise2018}. In the case of a right-angled Artin group  this result is orthogonal to our growth rate results as a subgroup that stabilizes a hyperplane in the universal cover of the Salvetii complex is never stable. However, examples in right-angled Coxeter groups show that it is possible for a hyperplane stabilizer in the Davis complex to be stable. Determining precisely when a hyperplane stabilizer of an cocompactly cubulated group is a stable subgroup appears to be an open and interesting question.
\end{problem}

\begin{problem}
Li and Wise show that when $G$ is the fundamental group of a compact special cube complex,  there is a sequence of infinite index quasi-convex subgroups
whose growth rates converge to the growth rate of $G$ \cite{LiWise2020}.  For a Morse local-to-global group $G$, does there exist  a sequence of infinite index stable subgroups whose growth rates converge to the growth rate of $G$? To our knowledge, this question is  open even for the case of right-angled Artin groups.
\end{problem}

\begin{problem}
 The inclusion of a  hyperbolic subgroup into a hyperbolic group induces a \emph{Cannon--Thurston map} if the  inclusion induces a well-defined continuous map between their Gromov boundaries; see \cite{Mitra1998,Elizabeth}. A surprising  theorem  of Mj proved that Cannon--Thurston maps exist  for any normal hyperbolic subgroup of a hyperbolic group \cite{Mitra1998}. 
 On the other hand, we show that the inclusion of a normal hyperbolic subgroup into a non-hyperbolic Morse local-to-global group never induces an analogous Cannon--Thurston map between their Morse boundaries; see Corollary \ref{cor:no_cannon-thurston} or \cite{BCGGS18} for the case of CAT(0) groups with isolated flats.
 Possibly with aid of Theorem \ref{intro_thm:regular_language_for_Morse_Geodesics},  determine when (if ever) the inclusion of an infinite index non-hyperbolic finitely generated normal subgroup in a non-hyperbolic Morse local-to-global group induces a Cannon--Thurston map for the Morse boundary.
\end{problem}

\subsection{Outline of the paper} Section \ref{sec:background} gives the necessary background information on languages, growth functions and rates, and Morse local-to-global groups.  Section \ref{sec:reg_language_for_Morse_geodesics} contains our construction of automatic structures for Morse geodesic words in Morse local-to-global groups (Theorem \ref{intro_thm:regular_language_for_Morse_Geodesics} ). Section \ref{sec:reg_language_for_stable_subgroups} applies these structures to build automatic structures for stable subgroups of Morse local-to-global groups (Theorem \ref{intro_thm:stable_subgroups}). Section \ref{sec:growth_rate_gap} uses the automatic structures for stable subgroups to prove our growth rate results (Theorems \ref{intro_thm:growth_torsion_free} and \ref{intro_thm:growth_res_finite}). Section \ref{sec:boundary_applications} contains our applications to the Morse boundary (Theorem \ref{intro_thm:boundary_applications}).

\subsection*{Acknowledgments} Cordes and Zalloum would like to thank the Forschungsinstitut f\"ur Mathematik (FIM) at ETH Z{\"u}rich  for supporting Zalloum's visit to ETH where this work was initiated; they would also like to extend a special thank you to Andrea Waldburger at the FIM for helping Zalloum obtain a visa to visit ETH.  Zalloum would like to thank Matthew Haulmark and Ilya Kapovich  for fruitful discussions. 
Russell would like to thank Alessandro Sisto for supporting his visit to ETH where this work was initiated.
Cordes was partially supported by the ETH Z{\"u}rich Postdoctoral Fellowship Program, cofunded by a Marie Curie Actions for People COFUND Program. 
Spriano was partially supported by the Swiss National Science Foundation (grant {\#}182186). We are very grateful to the anonymous referee for their many thoughtful comments and for asking a question that lead to Corollary \ref{cor:no_cannon-thurston}.

\section{Background}\label{sec:background}
\subsection{Languages, groups, and finite state automata}

The core objects of this paper are languages with finite alphabets.

\begin{definition}[A language with a finite alphabet]
Let $A$ be a finite set and $A^\star$ be the free monoid over $A$. We call an element $w \in A^\star$ a \emph{word} in $A$. If $w \in A^\star$ and $w = a_1\cdots a_n$ where each $a_i \in A$, then we call $a_1,\dots,a_n$ the \emph{letters} of the word $w$.  A \emph{language with alphabet $A$} is a set of words in $A^\star$.  The \emph{word length} of  $w \in A^\star$ is the number of letters of $A$ in the word $w$. We denote the word length of $w$ by $\ell(w)$.
\end{definition}

The alphabets for the languages we will be working with will be  finite generating sets for groups.

\begin{definition}
Let $G$ be a finitely generated group and let $A$ be a finite, symmetric generating set for $G$. The \emph{Cayley graph of $G$ with respect to $A$}, $\cay(G,A)$, is the graph whose vertices are the elements of $G$ and $g,h \in G$ are joined  by an edge if $g^{-1}h \in A$. If $g^{-1}h \in A$, then we label the edge connecting $g$ and $h$ by $g^{-1}h$.  The Cayley graph $\cay(G,A)$ is a metric space by declaring each edge to have length $1$. 
\end{definition}

\begin{convention}
Henceforth, we will assume that every finite generating set for a group is symmetric, that is, $A = A^{-1}$.
\end{convention}

When the group $G$ is generated by the finite set $A$, then every path in $\cay(G,A)$ produces a word in $A^\star$ by concatenating the labels of the edges in the order they appear along the path. Conversely, every word in $A^\star$ produces a path in $\cay(G,A)$ by starting at the identity $e$ and traversing the edges in $\cay(G,A)$ labeled by the letters of the word appearing from left to right. We will be particularly concerned with words in $A^\star$ that correspond with geodesic paths in $\cay(G,A)$.

\begin{definition}
For a word $w \in A^\star$, let $\overline{w}$ denote the element of $G$ obtained by viewing $w$ as an element of $G$.
A word $w \in A^\star$ is \emph{geodesic} in $\cay(G,A)$ if the path in $\cay(G,A)$ from $e$ to $\overline{w}$ labeled by the letters of $w$ is a geodesic in $\cay(G,A)$. This is equivalent to saying that $\ell(w)$ is minimal among all words that represent $\overline{w}$. For a word $w \in A^\star$, we define the \emph{geodesic word length} of $w$ to be $\ell(\overline{w})$ and denote it $|w|$. Similarly,  for $g \in G$, we define the \emph{geodesic word length} of $g$, denoted $|g|$, to be the geodesic word length of any $w \in A^\star$ with $\overline{w} = g$. 
\end{definition}

For our purposes, a finite state automaton is best understood as a directed graph  with a finite set labeling the edges and the vertices partitioned into an ``accept'' and  a ``reject'' set.

\begin{definition}[Finite state automaton]
Let $A$ be a finite set. A \emph{finite state automaton with alphabet $A$} is a tuple $\mc{G} = (\Gamma,A,Y,s_0)$ where \begin{itemize}
    \item $\Gamma$ is a finite directed graph. We call the vertices of $\Gamma$ the \emph{states} of $\mc{G}$.
    \item Each edge of $\Gamma$ is labeled by an element of  the alphabet $A$.
    \item $Y$ is a subset of the vertices of $\Gamma$.  The vertices of $Y$ are called the \emph{accept states} of $\mc{G}$, while the vertices not in $Y$ are called the \emph{reject states}.
    \item $s_0$ is a vertex of $\Gamma$. We call $s_0$ the \emph{initial state} of $\mc{G}$.
\end{itemize}
We will often conflate a finite state automaton with its underlying graph. For example, we may refer to the edges or vertices of $\mc{G}$ instead of $\Gamma$.
\end{definition}

Every finite state automaton produces an \emph{accepted language} consisting of words produced by reading the labels of each path in the directed graph that starts at the initial state and ends at one of the accept states. These languages are precisely the ones that are regular.

\begin{definition}[Regular Language]
A \emph{directed path} in a finite state automaton $\mc{G} = (\Gamma,A,Y,v_0)$ is a sequence of edges $e_1,\dots , e_n$ of $\Gamma$ so that the terminal vertex of $e_i$ is the initial vertex of $e_{i+1}$ for each $i \in\{1,\dots,n-1\}$. A word $w=a_1\cdots a_n$ in $A$ is \emph{read} by a path  $e_1,\dots,e_n$ in $\mc{G}$ if the label of $e_i$ is  $a_i$ for each $i\in\{1,\dots,n\}$. The \emph{language accepted by $\mc{G}$} is the set of words in $A$ that are read by paths in $\mc{G}$ that start at the initial state $s_0$ and end at an accept state of $\mc{G}$. A language is \emph{regular} if it is the accepted language of some finite state automaton.
\end{definition}

\subsection{Growth of languages and subgroups}
A basic feature of a language is the number of words of a specified word length. This information is often compiled as a growth function.

\begin{definition}[Growth function of a language]
Given a language $L$ with alphabet $A$, the \emph{growth function} of $L$ is the function $f_L(n) \colon \mathbb{N} \to \mathbb{N}$ where \[f_L(n) = | \{w \in L \mid \ell(w) \leq n\}|.\]
\end{definition}

While the growth function of a language captures the pure size of the language, the following growth rate captures the speed at which the language grows relative to the word length. 

\begin{definition}[Growth rate of a language]
Given a language $L$ with alphabet $A$, the \emph{growth rate} of $L$ is denoted $\lambda_L$ and defined to be \[\lambda_L = \limsup_{n \to \infty} \sqrt[n]{f_L(n)}.\]
\end{definition}

For regular languages there is a robust set of tools to understand the growth function of the language. One  foundational fact is that regular languages always have rational growth.

\begin{theorem}[\cite{Word_Processing}]\label{thm: rational growth of languages}
If $L$ is a regular language, then  there exist polynomials $P(x),Q(x) \in \mathbb{Q}[x]$ so that \[\sum_{n=0}^\infty f_L(n) \cdot x^n = \frac{P(x)}{Q(x)}. \]
\end{theorem}

Having rational growth is not a comment on the size of
the language, but rather on its simplicity. In fact,  a sequence
of integers $a_n$ has a rational generating function if and only if $a_n$ satisfies a
finite linear recursion \cite{Word_Processing}.

Dahmani, Futer, and Wise proved a folk theorem that Perron--Frobenius theory  can be used to calculate the growth rate of a regular language \cite{DahmaniFuterWise2018}. The motivation for turning to Perron--Frobenius theory  is the observation that if $L$ is a regular language accepted by the automaton $\mc{G}$, then the number of words in $L$ of length at most $n$ is precisely the number of directed paths of length at most $n$ in $\mc{G}$ that start at the initial state and end at some accept state. Thus, one should expect a connection between the growth rate of $L$ and the adjacency matrix of $\mc{G}$.

\begin{definition}[The Perron--Frobenius eigenvalue]
Given a finite directed graph $\Gamma$ with vertices $v_1,\dots,v_n$, the \emph{adjacency matrix for $\Gamma$} is the the $(n\times n)$-matrix whose $(i, j)$-th entry is $1$ if there exists a directed edge connecting the vertex $v_i$ to the vertex $v_j$ and 0 otherwise. The \emph{adjacency matrix for a finite state automaton} $\mc{G} = (\Gamma, A,Y,s_0)$ is the adjacency matrix for the directed graph $\Gamma$.  For a directed graph $\Gamma$ or finite state automaton $\mc{G}$, we define $\rho_\Gamma$ or $\rho_\mc{G}$ to be the  eigenvalue of the adjacency matrix for $\Gamma$ or $\mc{G}$ with the largest absolute value. We call $\rho_\Gamma$ and $\rho_\mc{G}$  the \emph{Perron--Frobenius eigenvalues} of $\Gamma$ and $\mc{G}$ respectively.
\end{definition}

To be able to fully exploit the connection between the adjacency matrix of an automaton $\mc{G}$ and the growth function of the accepted language of $\mc{G}$, we need to require that our automaton  satisfies a minor minimality condition.

\begin{definition}[Pruned automaton]
A finite state automaton is \emph{pruned} if every state is the vertex of some path from the initial state to an accept state. 
\end{definition}

\begin{remark}[Pruning an automaton]
 Given a finite state automaton $\mc{G}$, we can produce a pruned finite state automaton $\mc{G}'$ that has the same accepted language as $\mc{G}$ by deleting all of the states of $\mc{G}$ that do not appear along a path from the initial state to an accept state.
\end{remark}

When the finite state automaton is pruned, Dahmani, Futer, and Wise used classical  results from Perron--Frobenius theory to show that the growth rate of the accepted language of the automaton is equal the Perron--Frobenius eigenvalue of the automaton.

\begin{theorem}[{\cite[Theorem 3.6]{DahmaniFuterWise2018}}]\label{thm:growth rate is perron-frobienius}
Let $\mc{G}$ be a pruned finite state automaton that accepts the regular language $L$. We have $\rho_\mc{G} \geq 1$ and \[\rho_\mc{G} = \lim\limits_{n\to \infty} \sqrt[n]{f_L(n)} = \lambda_L.\]
 
\end{theorem}

We also need the following lemma that combines Lemma 3.2 and Theorem 3.4 of \cite{DahmaniFuterWise2018}. When combined with Theorem \ref{thm:growth rate is perron-frobienius}, this allows us to conclude that the growth rate of certain sublanguages of regular languages are strictly smaller than the growth rate of the ambient language. 

\begin{lemma}[{\cite[Lemma 3.2 and Theorem 3.4]{DahmaniFuterWise2018}}]\label{lem:proper subgraphs have smaller eigenvalue}
Let $\Gamma$ be a directed graph and let $\Gamma'$ be a proper subgraph of $\Gamma$. If for every pair of distinct vertices $v,w$ of  $\Gamma$, there exits a directed path in $\Gamma$ from $v$ to $w$ and from $w$ to $v$, then $\rho_{\Gamma'} < \rho_{\Gamma}.$
\end{lemma}

Ultimately, our present interest in the growth of languages is to deduce facts about the growth rate of subgroups of a finitely generated group.

\begin{definition}(Growth rate of a subgroup)\label{defn:growth rate of a subgroup}
Let $G$ be a finitely generated group with finite generating set  $A$. For $g \in G$ and $r \geq 0$, let $B(g,r)$ be the set $\{h \in G \mid |h^{-1}g| \leq r\}$. For a subgroup  $H \leq G$, the \emph{growth function} of $H$ with respect to $A$ is the function $f_{H,A} \colon \mathbb{N} \to \mathbb{N}$ where \[f_{H,A}(n) = |B(e,n) \cap H|.\] The \emph{growth rate} of $H$ with respect to $A$ is denoted $\lambda_{H,A}$ and defined as \[\lambda_{H,A} = \limsup_{n\to \infty} \sqrt[n]{f_{H,A}(n)}.\]
\end{definition}

 The next corollary uses Theorem \ref{thm:growth rate is perron-frobienius} to relate the growth rate of a subgroup $H$ with  the growth rate of a regular language composed of geodesic words in the group that are in bijection with $H$.
 
 \begin{definition}
 Let $G$ be a finitely generated group with a finite generating set $A$, and let $L$ be a regular language with alphabet $A$. We say $L$ is a \emph{geodesic language} if for all $w \in L$, $w$ is a geodesic word in $\cay(G,A)$. We say $L$ \emph{bijects with a subgroup} $H\leq G$ if for each $w \in L$, $\overline{w} \in H$ and the map $L \to H$ given by $w \to \overline{w}$ is a bijection.
 \end{definition}
 
\begin{corollary} \label{Sequence of equalities}
Let $G$ be a finitely generated group with a finite generating set $A$. Suppose there is a regular, geodesic language $L$ with alphabet $A$ that bijects with a subgroup $H \leq G$.
If $\mc{G}$ is a pruned finite state automaton that accepts $L$, then \[ \lambda_{H,A} = \lambda_L = \rho_{\mc{G}}.\]
\end{corollary}

\begin{proof}

 Since the language $L$ is a language of geodesics words in $\cay(G,A)$ that biject to $H$, we have $f_{H,A}=f_L$. If $\mc{G}$ is any pruned finite state automaton that accepts $L$, Theorem \ref{thm:growth rate is perron-frobienius} implies

\[  \lambda_{H,A} = \underset{n \rightarrow \infty}{\lim} \text{sup} \sqrt[n]{f_{H,A}(n)}=\underset{n \rightarrow \infty}{\lim} \text{sup} \sqrt[n]{f_{L}(n)}=\lambda_L=\underset{n \rightarrow \infty}{\lim}  \sqrt[n]{f_{L}(n)}=\rho_\mc{G} .\qedhere\]
\end{proof}

\subsection{Morse quasi-geodesics, Morse local-to-global groups, and stable subgroups}
We now define the class of groups we will focus on in this paper---Morse local-to-global groups. For the rest of this section, let $I$ denote a closed interval of $\mathbb{R}$ and $X$ a metric space.

\begin{definition}
Let $M \colon [1,\infty) \times [0,\infty) \to [0,\infty)$ and $B \geq 0$. The quasi-geodesic $\gamma \colon I \to X$ is an  \emph{$M$-Morse quasi-geodesic} if for any interval $[s,t] \subseteq I$, any $(k,c)$-quasi-geodesic with endpoints $\gamma(s)$ and $\gamma(t)$ is contained in the $M(k,c)$-neighborhood of $\gamma\vert_{[s,t]}$.
If $\gamma$ is an $M$-Morse $(k,c)$-quasi-geodesic, we say $\gamma$ is a $(M;k,c)$-Morse quasi-geodesic. A map $\gamma \colon I \to X$ is a \emph{$(B;M;k,c)$-local Morse quasi-geodesic} if for any $[s,t]\subseteq I$ we have \[|s-t| \leq B \implies \gamma\vert_{[s,t]} \text{ is a } (M;k,c)\text{-Morse quasi-geodesic}.\] If $\gamma$ is a $(B;M;1,0)$-local Morse quasi-geodesic, we say $\gamma$ is a \emph{$(B;M)$-local Morse geodesic}. The function $M$ is called a \emph{Morse gauge} and the number $B$ is called the \emph{local scale} of the local quasi-geodesic.
\end{definition}

In a finitely generated group with a finite generating set $A$, we label words in $A$ as Morse or locally Morse if the corresponding paths in the Cayley graph are Morse or locally Morse.

\begin{definition}
Let $G$ be a group with finite generating set $A$. Let $M$ be a Morse gauge and $B \geq 0$. We say a geodesic word $w \in A^\star$ is \emph{$M$-Morse} if the geodesic in $\cay(G,A)$ from $e$ to $\overline{w}$ labeled by $w$ is $M$-Morse. Similarly, we define a word $w \in A^\star$ to be $(B;M)$-locally Morse if the path in $\cay(G,A)$ from $e$ to $\overline{w}$ labeled by $w$ is a $(B;M)$-local Morse geodesic.
\end{definition}

While Morse quasi-geodesics share many similar properties with quasi-geodesics in hyperbolic spaces, one fundamental property that is not carried over from hyperbolic spaces is a local-to-global property that ensures local Morse quasi-geodesics of sufficient local scale are actually global Morse quasi-geodesics; see \cite{Morse-local-to-global} for examples. Morse local-to-global spaces are therefore spaces where this local-to-global property for Morse quasi-geodesics does hold.

\begin{definition}[{\cite[Definition 2.12]{Morse-local-to-global}}]\label{defn:Morsel_local_to_global}
A metric space $X$ is a \emph{Morse local-to-global space} if for every Morse gauge $M$ and constants $k\geq 1$, $c \geq 0$, there exists a local scale $B \geq 0$, Morse gauge $M'$, and constants $k' \geq 1$, $c'\geq 0$ so that every $(B;M;k,c)$-local Morse quasi-geodesic is a global $(M';k',c')$-Morse quasi-geodesic. A finitely generated group $G$ is a \emph{Morse local-to-global group}, if there exists a finite generating set $A$ for $G$ so that $\cay(G,A)$ is a Morse local-to-global space. This definition is independent of choice of finite generating set since being a Morse local-to-global space is a quasi-isometry invariant.
\end{definition}

Russell, Spriano, and Tran proved that a large number of interesting, non-hyperbolic groups have the Morse local-to-global property.

\begin{theorem}[{\cite[Theorems D and E]{Morse-local-to-global}}]\label{thm:Morse local-to-global groups}
The following groups have the Morse local-to-global property.
\begin{itemize}
    \item All $\operatorname{CAT}(0)$ groups.
    \item All hierarchically hyperbolic groups.
    \item The mapping class group of an orientable,  finite type surface.
    \item The fundamental group of any compact $3$-manifold.
    \item All unconstricted groups, such as, solvable groups and any group with infinite order central element.
    \item Any group hyperbolic relative to groups with the Morse local-to-global property.
\end{itemize}
\end{theorem}

The Morse local-to-global property is particularly useful in studying stable subgroups of Morse local-to-global groups. Stable subgroups are the  generalization of quasi-convex subgroups of hyperbolic groups that accompanies Morse geodesics.

\begin{definition}\label{def:stable_subgroup} (Stable subgroup). Let $G$ be a group with finite generating set $A$, $M$ be a Morse
gauge, and $k \geq 0$. A subgroup $H \leq G$ is \emph{$(M, k)$-stable} in Cay$(G,A)$, if for every 
$h\in H$, every geodesic in $\cay(G,A)$ from $e$ to $h$ is $M$-Morse and contained in the $k$-neighborhood of $H$. A
subgroup $H \leq G$ is a \emph{stable subgroup} if for any choice of finite generating set $A$ for $G$, there exist 
$M$ and $k$ such that $H$ is $(M, k)$-stable in $\cay(G,A)$.
\end{definition}

Stable cyclic subgroups  are a particularly important class of stable subgroups because they produce Morse quasi-geodesics in the Cayley graph of the group. The generators of such subgroups are often called Morse elements.

\begin{definition}[Morse element]
Let $G$ be group with a finite generating set $A$. We say $g \in G$ is \emph{Morse} if $\langle g \rangle$ is a stable subgroup of $G$.
\end{definition}

Note, an element being Morse is a much stronger statement than an element being represented by a Morse geodesic word. If $g$ is represented by some $M$-Morse geodesic word, there is no guarantee that a power of $g$ is also represented by an $M$-Morse geodesic word. However, if $g$ is a Morse element of $G$, then every power of $g$ is represented by a Morse geodesic word with the same Morse gauge.

\section {Regular languages for Morse geodesic words}\label{sec:reg_language_for_Morse_geodesics}

The goal of this section is to construct automatic structures for the Morse geodesic words in a Morse local-to-global group. Given a Morse local-to-global group, we will construct a regular language $L_M$ for each Morse gauge $M$ so that every $M$-Morse geodesic word is an element of $L_M$ and every element of $L_M$ is an $M'$-Morse geodesic word. We define $L_M$ to be the set of geodesic words that are \emph{locally} $M$-Morse. The Morse local-to-global property ensures these geodesics are $M'$-Morse for $M'$ determined by $M$.

\begin{definition}[The language $L_M$]\label{defn:L_M}
Let $G$ be a Morse local-to-global group with finite generating set $A$. For a Morse gauge $M$,  let $B_M \geq 0$ be the local scale so that each $(B_M;M)$-local Morse geodesic is a $(M';k,c)$-Morse quasi-geodesic where $M'$, $k$, and $c$ depend only on $M$. Define $L_M$ to be  the language of all geodesic words in $A^\star$ that are $(B_M;M)$-locally Morse in $\cay(G,A)$.
\end{definition}

By of the end of the section, we will prove the following theorem that implies Theorem \ref{intro_thm:regular_language_for_Morse_Geodesics} from the introduction as every $M$-Morse geodesic word is  an element of $L_M$.

\begin{theorem}[Regular language for Morse geodesics]\label{thm:regular_language_for_Morse_geodesics}
Let $G$ be a Morse local-to-global group with a finite generating set $A$ and let $L_M$ be the language from Definition \ref{defn:L_M}.
\begin{enumerate}
    \item There exists a Morse gauge $M'$ determined by $M$ so that every element of $L_M$ is an $M'$-Morse geodesic word. \label{item:reg_language_1}
    \item For each Morse gauge $M$, the language $L_M$ is regular. \label{item:reg_language_2}
    \item  \label{item:reg_language_3} Let $\preceq$ be a total order order on $A$ and extend $\preceq$ to a lexicographic ordering on $A^\star$. The  language $J_M \subseteq L_M$ defined by
    \begin{center}
    $J_M= \{u \in L_M \mid  \text{whenever } v \in  L_M \text{ with } \overline{u}=\overline{v} \text{ we have } u \preceq v      \}$
\end{center} 
is regular and has the property: if $u,v \in J_M$ with $\overline{u}= \overline{v}$ in $G$, then $u=v$. 
\end{enumerate}
\end{theorem}

\begin{proof}
By choice of the local scale $B_M$, item  (\ref{item:reg_language_1}) follows from applying the Morse local-to-global property of $\cay(G,A)$ to the geodesics representing elements of $L_M$. Items (\ref{item:reg_language_2}) and (\ref{item:reg_language_3})  are proved in Corollaries \ref{cor: Morse geodesic form regular langauge} and \ref{cor: lexographically least Morse geodesics} respectively.
\end{proof}
 
Subsection \ref{sec:Morse_cones} introduces  the objects  and preliminary lemmas that will be used in the proof of the regularity of $L_M$; the proof of regularity  appears in Subsection \ref{sec:finite_Morse_cone_type}. Subsection \ref{sec:unique_representatives} uses the regularity of $L_M$ to prove the regularity of the short-lex sublanguage $J_M$.

\subsection{Morse cone types}\label{sec:Morse_cones}
In the setting of hyperbolic groups, the key to Cannon's proof that the set of geodesic words form a regular language is to use the local-to-global property of quasi-geodesics to prove that a hyperbolic group has a finite number of \emph{cone types} \cite{Cannon1984}.  
Our proof of  that $L_M$ is regular will closely follow Cannon's proof in the hyperbolic case, however some care must be taken when defining our analogue of Cannon's cone types. Cannon defined the cone type of a geodesic word $u$ to be all the words $w$ so that $uw$ is  a geodesic word. The most straight forward generalization of this definition to Morse geodesics would be to define an $M$-Morse cone of a geodesic word $u$ to be all of the geodesic words $w$ so the $uw$ is an $M$-Morse geodesic word. However, this definition does not permit the appropriate level of local control necessary to prove there are a finite number of cone types. We therefore use the following definition of an $M$-Morse cone that requires the combined word $uw$ to be a geodesic that is locally $M$-Morse for a sufficient scale.

\begin{definition}[Morse cones and cone types]\label{defn:Morse_cones}
Let $M$ be a Morse gauge and $G$ be a Morse local-to-global group with finite generating set $A$. For a geodesic word $u \in A^\star$, define the \emph{$M$-Morse cone} of $u$ to be all the words $w \in A^\star$ so that $uw$ is a geodesic word in $\cay(G,A)$ that is $(B_M;M)$-locally Morse in $\cay(G,A)$. We denote the $M$-Morse cone of $u$ by $\Cone^M(u)$, and we say that $u$ and $v$ have the same \emph{$M$-Morse cone type} if the sets $\Cone^M(u)$  and $\Cone^M(v)$ are equal.
\end{definition}

In the hyperbolic case, Cannon proves that there is a number $N$ so that if the cones types for two elements agree for all elements with word length at most $N$, then the cone types are equal. This is a ``local-to-global'' property for cone types. In our case, we will also show that the $M$-Morse cone is determined ``locally'', however we will need to keep track of two pieces of local information, both the restriction of the $M$-Morse cones to words of length at most $N$ and the  tails of  elements of length at most $N$.

\begin{definition}[Local restriction of the cone]\label{def: Cay contracting type}
For a geodesic word $u \in A^\star$, $N \in \mathbb{N}$, and Morse gauge $M$, define the \emph{$N$-restricted $M$-Morse cone} of $u$ to be all words $w \in \Cone^M(u)$ with $|\overline{w}|\leq N$. We denote the $N$-restricted $M$-Morse cone of $u$ by $\Cone_N^M(u)$.
\end{definition}

\begin{definition}[$N$-tails of an element]\label{def: Cay k-tail}
Given a geodesic word $u \in A^\star$ and $N\in \mathbb{N}$, define the \emph{$N$-tail} of $u$ to be all  elements $g \in G$ with $|g| \leq N$ and $|\overline{u}g| < |\overline{u}|$. We denote the $N$-tail of $u$ by $T_N(u)$.
\end{definition}

We now record two lemmas we will need in our proof of  the regularity of $L_M$. The first says any geodesic that starts and ends close to a Morse geodesic must be a Morse geodesic while the second implies such pairs of geodesics must fellow travel for a significant distance.

\begin{lemma}[{\cite[Lemma 2.1]{Cordes2017}}]\label{lem: Morse is contagious}
Let $X$ be a  geodesic metric space and let $\alpha, \beta \colon [0,s] \rightarrow X$ be finite geodesics with $\alpha(0)=\beta(0)$ and $d(\alpha(s), \beta(s)) \leq 1$. If $\alpha$ is $M$-Morse, then there exists $M' \geq M$, depending only on $M$, such that $\beta$ is $M'$-Morse.
\end{lemma}

\begin{lemma}[{\cite[Lemma 2.7]{Cordes2017}}]\label{lem: fellow traveling}
Let $X$ be a geodesic metric space. If $\alpha, \beta \colon [0,s] \rightarrow X$ are $M$-Morse geodesics with $\alpha(0) = \beta(0)$ and $d(\alpha(r), \beta)<K$ for some $ r \in [0,s]$ and some $K>0$, then $d(\alpha(t), \beta(t))\leq8M(3,0)$ for all $t<r-K-4M(3,0)$. In particular, if $\alpha(s) = \beta(s)$, then $d(\alpha(t),\beta(t)) \leq 12M(3,0)$ for all $t \in [0,s]$.
\end{lemma}

We finish by noting a consequence of the above lemmas that we shall use in the proof of the regularity of $L_M$.
\begin{remark}\label{rem:Morse gauge for fellow traveling}
Let $G$ be a Morse local-to-global group with finite generating set $A$. If  $\alpha$ is a geodesic that is $(B_M;M)$-locally Morse in $\cay(G,A)$, then, by the Morse local-to-global property, $\alpha$ is $M'$-Morse where $M'$ depends only on $M$ and $A$. By Lemma \ref{lem: Morse is contagious}, if $\beta$ is any other geodesic in the Cayley graph sharing a start point with $\alpha$ and with end point 1 away from $\alpha$'s endpoint, then there exists a Morse gauge $M''$, depending only on $M'$, such that $\beta$ is $M''$-Morse. Lemma \ref{lem: fellow traveling} therefore gives us that $d(\alpha (t), \beta(t))<8M''(3,0)$ for all $t<|\alpha|-1-4M''(3,0)$.
\end{remark}

\subsection{The regularity of $L_M$} \label{sec:finite_Morse_cone_type}

As discussed in the previous section, the main step in proving  the regularity of $L_M$ is proving that any Morse local-to-global group has only finitely many Morse cone types for each Morse gauge.

\begin{theorem} \label{thm: finite many cone types}
Let $G$ be a finitely generated group with finite generating set $A$. If $G$ is a Morse local-to-global group, then for each Morse gauge $M$, there are only finitely many $M$-Morse cone types of geodesic words in $\cay(G,A)$.
\end{theorem}

\begin{proof}
Let $M$ be a Morse gauge. Recall, $B_M \geq 0$ is the local scale so that geodesics  in $\cay(G,A)$ that are $(B_M,M)$-locally Morse are globally $M'$-Morse.  Let $M''$ be the Morse gauge from Remark \ref{rem:Morse gauge for fellow traveling} determined by $M$ and $\cay(G,A)$.  For each $N \geq 0$, there are only finitely many possible $N$-tails and $N$-restricted $M$-Morse cones for all geodesic words in $\cay(G,A)$. Thus, to prove $\cay(G,A)$ contains only a finitely many $M$-Morse cone types, it suffices to prove the following claim.

\begin{claim}\label{claim:same_tails_imply_same_cones}
There exists $m$ depending only on $M$, $A$, and $G$ so that if two geodesic words $u,v \in A^\star$ have $T_m(u) = T_m(v)$ and $\Cone^M_m(u) = \Cone^M_m(v)$, then $\Cone^M(u) = \Cone^M(v)$.
\end{claim}

Let $m$ be the first integer greater than or equal to $\max\{1+8M''(3,0), B_M+1\}$ and assume $u,v \in A^\star$ are geodesic words in $\cay(G,A)$ with  $T_m(u) = T_m(v)$ and $\Cone^M_m(u) = \Cone^M_m(v)$. To prove Claim \ref{claim:same_tails_imply_same_cones}, we will prove by induction that $\Cone^M_N(u) = \Cone^M_N(v)$ for all $N \in \mathbb{N}$.

If $N \leq m$, then $ \Cone^M_N(u) = \Cone^M_N(v)$ by assumption. Assume that $\Cone^M_N(u) = \Cone^M_N(v)$ for some $N > m$. We will show this implies $\Cone^M_{N+1}(u) = \Cone^M_{N+1}(v)$.  Note, it suffices to only prove $\Cone^M_{N+1}(v) \subseteq \Cone^M_{N+1}(u)$ as a mirrored argument will establish the equality.

Let $q \in \Cone^M_{N+1}(v)$. Thus, there exists $w \in \Cone^M_{N}(v)$ and $a \in A^\star$ such that $q = wa$. By the induction hypothesis, $w \in \Cone^M_N(u)$. 

We first prove $uq = uwa$ is a geodesic word in $\cay(G,A)$. This part of the proof closely follows \cite{EikeZalloum} and the cone types section of \cite{BH}. Recall, for $\omega \in A^\star$, $\ell(\omega)$ denotes the word length in the free monoid $A^\star$, while $| \overline{\omega}|$ denotes the length of a geodesic word  in $\cay(G,A)$ that represents the group element $\overline{\omega}$. In particular $\ell(\omega) = |\overline{\omega}|$ if and only if $\omega$ is a geodesic word in $\cay(G,A)$.

For the purposes of contradiction, suppose $uq = uwa$ is not a geodesic word. Then, there must exist some geodesic word $p$ of length strictly less than $\ell(u)+\ell(w)+1$ such that $\overline{p}=\overline{uwa}$. Write $p$ as a product $p_1p_2$ such that $\ell(p_1)= \ell(u)-1$ and $\ell(p_2) \leq \ell(w)+1.$  Since $w \in \Cone^{M}(u)$,  $uw$ is a geodesic word that is $(B_M;M)$-locally Morse.  Thus, $uw$ is $M'$-Morse because $G$ has the Morse local-to-global property. Since the geodesics in $\cay(G,A)$ starting at $e$ and labeled by $uw$ and $p$ end 1 apart from each other, Remark \ref{rem:Morse gauge for fellow traveling} gives us that $d(\overline{p}_1,\overline{u})<8M''(3,0)+1$. Define $z=u^{-1}p_1$, hence, the group element $\overline{z}=\overline{u^{-1}p_1}$ satisfies $|\overline{z}| \leq 8M''(3,0)+1 \leq m$ and $|\overline{uz}| <|\overline{u}|$. This implies that $\overline{z} \in T_m(u).$ Recall that $T_m(u)=T_m(v)$ by assumption, so $\overline{z} \in T_m (v)$ and $|\overline{vz}|<|\overline{v}|$. Let  $\nu$ be any geodesic word representing the group element $\overline{vz}$.
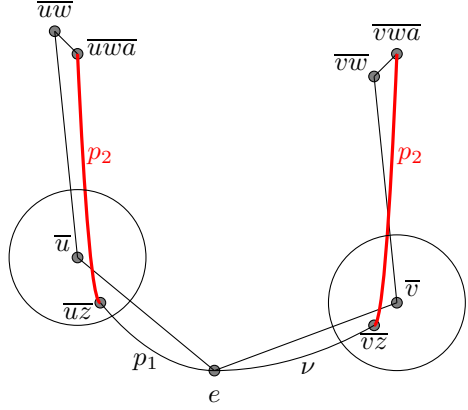
\begin{figure}[ht]
    \centering

    \begin{tikzpicture}[scale=.3]
        
        \coordinate (id) at (0,0);
        \coordinate (u) at (-6,5);
        \coordinate (v) at (8, 3);
        \coordinate (w) at (-1,10);
        \coordinate (a) at (1,1);
        \coordinate (uw) at (-7,15);
        \coordinate (vw) at (7,13);
        \coordinate (uwa) at (-6,14);
        \coordinate (vwa) at (8,14);
        \coordinate (uz) at (-5,3);
        \coordinate (vz) at (7,2);

        \draw (id) node[vertex,label={[shift={(0,-.5)}]$e$}]{};
        \draw (u) node[vertex,label={[shift={(-0.2,0)}]$\overline{u}$}]{};
        \draw (v) node[vertex,label={[shift={(.2,0)}]$\overline{v}$}]{};
        \draw (uw) node[vertex, label={[shift={(0,.1)}]$\overline{uw}$}]{};
        \draw (uwa) node[vertex, label={[shift={(.45,-.1)}]$\overline{uwa}$}]{};
        \draw (vw) node[vertex, label={[shift={(-.3,0)}]$\overline{vw}$}]{};
        \draw (vwa) node[vertex, label={[shift={(0,0.1)}]$\overline{vwa}$}]{};
        \draw (uz) node[vertex, label={[shift={(-.3,-.3)}]$\overline{uz}$}]{};
        \draw (vz) node[vertex, label={[shift={(0,-.4)}]$\overline{vz}$}]{};
        \draw (-5,9) node[label={\color{red}$p_2$}]{};
        \draw (8.6,9) node[label={\color{red}$p_2$}]{};
        \draw (4.1,-.25) node[label={$\nu$}]{};
        \draw (-3,-.1) node[label={$p_1$}]{};

        \draw (id)--(u);
        \draw (u)--(uw);
        \draw (uw)--(uwa);
        \draw (id)--(v);
        \draw (v)--(vw);
        \draw (vw)--(vwa);
        \draw (id) parabola (uz);
        \draw[very thick, color=red] (uz) parabola (uwa);
        \draw[very thick, color=red] (vz) parabola (vwa);
        \draw (id) parabola (vz);
        
        \draw (u) circle (3cm);
        \draw (v) circle (3cm);

    \end{tikzpicture}
     \caption{The path starting at $e$ and labeled by $\nu p_2$ ends at $\overline{vwa}$.}
    \label{fig: geodesic cone}
\end{figure}

 As shown in  Figure \ref{fig: geodesic cone}, the  path starting at $e$ and labeled by $\nu p_2$ ends at the vertex $\overline{vwa}$ because
\begin{center}
$\overline{\nu p_2}=\overline{vzp_2}=\overline{vu^{-1}p_1p_2}=\overline{vu^{-1}uwa}=\overline{vwa}.$
\end{center}
 Since $vwa$ is a geodesic word,  this implies $\ell(\nu p_2) \geq \ell(vwa) = \ell(v)+\ell(w)+1.$  On the other hand,  since $\ell(\nu) = | \overline{vz} |< | \overline{v} | =\ell(v)$ and $\ell(p_2) \leq \ell(w) +1$ we get  the contradiction

 \begin{center}
 $\ell(\nu p_2) \leq \ell(\nu)+\ell(p_2) < \ell(v)+\ell(w)+1$.
 \end{center}
 Therefore, $uq = uwa$ must be a geodesic word.

We now prove that $uq = uwa$ is $(B_M;M)$-locally Morse.  Let $\gamma$ be the path in $\cay(G,A)$ starting at $e$ and labeled by $uwa =uq$.  Since $\ell(wa) > m >  B_M$, any subpath of $\gamma$ with length at most $B_M$ must be contained in a subpath labeled by either $uw$ or $wa$. However, any path in $\cay(G,A)$ label by $uw$ or $wa$ is $(B_M;M)$-locally Morse since $w \in \Cone^M(u)$ and $wa \in \Cone^M(v)$ respectively. Thus, $\gamma$, and hence $uwa=uq$, is $(B_M;M)$-locally Morse.

Since $uq = uwa$ is a geodesic word that is $(B_M;M)$-locally Morse, $q \in \Cone^M_{N+1}(u)$ and $\Cone^M_{N+1}(v) \subseteq \Cone^M_{N+1}(u)$. Switching the role of $u$ and $v$ we have that $$\Cone^M_N(u) = \Cone^M_N(v) \implies \Cone^M_{N+1}(u) = \Cone^M_{N+1}(v)$$ as desired. Therefore, by induction, we have proved Claim \ref{claim:same_tails_imply_same_cones}, which in turn proves Theorem \ref{thm: finite many cone types}.
\end{proof}

Since there are only finitely many $M$-Morse cone types, we can build an automaton that reads the language $L_M$ using the $M$-Morse cone types as the vertices.

\begin{corollary}\label{cor: Morse geodesic form regular langauge}
Let $G$ be a Morse local-to-global group. For any finite generating set $A$ and any Morse gauge $M$, the language $L_M$ from Definition \ref{defn:L_M} is regular.
\end{corollary}

\begin{proof}
Let $\mc{G}$ be the following finite state automaton.
\begin{itemize}
    \item The states are $M$-Morse cone types of geodesic words in $A^\star$.  Theorem \ref{thm: finite many cone types} shows that there are finitely many states.
    \item There is a directed edge labeled by $a\in A$ connecting the $M$-Morse cone type of a geodesic word $u$ to the $M$-Morse cone type of $ua$ if and only if $a$ belongs to the $M$-Morse cone of $u$.
    \item The initial state is the $M$-Morse cone type of the identity word.
    \item All states of $\mc{G}$ are accept state.
\end{itemize}
The language $L_M$ is the accepted language of $\mc{G}$.
\end{proof}

\subsection{Improving $L_M$ to a language of unique geodesic representatives}\label{sec:unique_representatives}

The language $L_M$ from Definition \ref{defn:L_M} may contain multiple different geodesic words that represent the same group element. The goal of this subsection is  to refine the language 
$L_M$ to a regular language $J_M \subseteq L_M$ that contains a unique geodesic representative for each group element represented by a word in $L_M$.  The first step in this process is proving that the regularly of $L_M$ implies that that language of pairs of words in $L_M$ that represent the same element of $G$ is also regular. Note, we view the free monoid $(A \times A)^\star$ as a subset of the monoid $A^\star \times A^\star$.

\begin{lemma}[Equality recognizer] \label{lem: equality recognizer} Let $G$ be a Morse local-to-global group with finite generating set $A$, and let $M$ be a Morse gauge. The language $Q_M \subseteq (A\times A)^\star$ consisting of 
\begin{center}
$Q_M=\{(u,v) \in (A \times A)^\star \mid \text{ $u, v \in L_M$ and $\overline{u}=\overline{v}$}\}$
\end{center}
is a regular language.
\end{lemma}

\begin{proof} Let $M'$ be the Morse gauge determined by $M$ and $\cay(G,A)$ so that each word in $L_M$ is an $M'$-Morse geodesic word in $\cay(G,A)$; see Remark \ref{rem:Morse gauge for fellow traveling}. Let $r = 12M'(3,0)$.

By Corollary \ref{cor: Morse geodesic form regular langauge}, there exists a finite state automaton $\mc{G}$ that accepts the language $L_M$. Denote the vertices of this automaton by $V(\mc{G})$. We will use $\mc{G}$ to build a finite state automaton $\mc{Q}$ with alphabet $A\times A$  whose accepted language is  $Q_M$. Define $\mc{Q}$ as follows:

\begin{itemize}
    \item The vertex set for $\mc{Q}$ is $V(\mc{G})\times  V(\mc{G}) \times B(e,r)$ where $B(e,r)$ is the set of elements of $G$ whose geodesic word length is at most $r$.
    \item There is a directed edge from $(s,\sigma,g)$ to $(t,\tau,h)$ labeled by $(a,\alpha) \in A \times A$ if all of the following are satisfied.
    \begin{enumerate}
        \item There is a directed edge from $s$ to $t$ labeled by $a$ in $\mc{G}$.
        \item There is a directed edge from $\sigma$ to $\tau$ labeled by $\alpha$ in $\mc{G}$.
        \item $h = a^{-1} g \alpha$ is in  $B(e,r)$.
    \end{enumerate}
    \item The accept states of $\mc{Q}$ are vertices of the form $(s,\sigma,e)$ where $s$ and $\sigma$ are both accept states of $\mc{G}$.
    \item The initial state of $\mc{Q}$ is $(s_0,s_0,e)$ where $s_0$ is the initial state of $\mc{G}$.
\end{itemize}

A directed path in $\mc{Q}$ from the initial state $(s_0,s_0,e)$ to an accept state $(s,\sigma,e)$ will read a word $(w,\omega) \in (A \times A)^\star$ where $\overline{w} = \overline{\omega}$ in $G$ and $w,\omega \in L_M$. Thus, the accepted language of $\mc{Q}$ is contained in $Q_M$.

To see that the accepted language of $\mc{Q}$ is exactly $Q_M$, let $(w,\omega) \in Q_M$. Since $w,\omega \in L_M$, they are both geodesic words representing the same element of $G$. Thus, $\ell(w) = \ell(\omega)$.  Let $w = a_1 \cdots a_n$ and $\omega = \alpha_1 \cdots \alpha_n$ where each $a_i$ and $\alpha_i$ are elements of $A$. Since $L_M$ is the accepted language for $\mc{G}$, there are directed paths $\eta_w$ and $\eta_\omega$  in $\mc{G}$ that read $w$ and $\omega$ respectively. For $i \in \{1,\dots,n\}$, let $s_i$ (resp. $\sigma_i$) be the vertex of $\mc{G}$ that is reached by the subsegment of $\eta_w$ (resp. $\eta_\omega$) that is labeled by $a_1\cdots a_i$ (resp. $\alpha_1 \cdots \alpha_i$). Since any subwords of $w$ and $\omega$ are also geodesic words that are $(B_M;M)$-locally Morse, we have  that $s_i$ and $\sigma_i$ are accept states of $\mc{G}$ for each $i $. Since $w$ and $\omega$ are $M'$-Morse geodesic words in $\cay(G,A)$ with $\overline{w} = \overline{\omega}$, Lemma \ref{lem: fellow traveling} implies $| \overline{a}_i^{-1}\cdots \overline{a}_1^{-1} \overline{\alpha}_1 \cdots \overline{\alpha}_i| \leq r = 12M'(3,0)$ for each $i$. Thus, $(s_i,\sigma_i,\overline{a}_i^{-1}\cdots \overline{a}_1^{-1} \overline{\alpha}_1 \cdots \overline{\alpha}_i)$ is a vertex of $\mc{Q}$ for each $i$. Hence, there is a path from $(s_0,s_0,e)$ to $(s_n,\sigma_n,e)$ in $\mc{Q}$ establishing that $(w, \omega)$  is in the accepted language of $\mc{Q}$.
\end{proof}

Since the language $Q_M$ in Lemma \ref{lem: equality recognizer} is regular, we can  invoke a special case of a proposition of Baumslag, Gersten, Shapiro, and Short to refine $L_M$.

\begin{proposition}[{\cite[Proposition 5.5]{BaumslagGerstenShapiroShort1991}}] \label{prop: lexographically least language}
Let $A$ be a finite set with a total order $\preceq$. Extend $\preceq$ to a lexicographic ordering on $A^\star$. If $Q$ is a regular language with alphabet $A \times A$, then
\begin{center}
    $J=\{u \in A^\star \mid (u,v) \in Q \text{ for some $v$}\text{ and $(u,w) \in Q \implies u \preceq w$} \}$
\end{center} 
is a regular language with alphabet $A$.

\end{proposition}

\begin{corollary}[Short-lex sublanguage]

\label{cor: lexographically least Morse geodesics}
Let $G$ be a Morse local-to-global group with finite generating set $A$, and let $M$ be a Morse gauge. Let $\preceq$ be a lexicographic ordering on $A^\star$ induced by a total order on $A$. The language $J_M \subseteq L_M$ defined as 
\begin{center}
    $J_M= \{u \in L_M \mid  \text{whenever } v \in  L_M \text{ with } \overline{u}=\overline{v} \text{ we have } u \preceq v      \}$
\end{center} 
is regular and has the property: if $u,v \in J_M$ with $\overline{u}= \overline{v}$ in $G$, then $u=v$.
\end{corollary} 

\begin{proof}
Let $Q_M$ the regular language from Lemma \ref{lem: equality recognizer}. The language $J_M$ is precisely 
\begin{center}
    $J_M=\{u \in A^\star \mid (u,v) \in Q_M \text{ for some $v$}\text{ and $(u,w) \in Q_M \implies u \preceq w$} \}$.
\end{center}
Since the language $Q_M$ is regular, Proposition \ref{prop: lexographically least language} implies that $J_M$ is also regular. If $u,v \in J_M$ with $\overline{u}= \overline{v}$, then $u=v$ since $u \preceq v$ and $v \preceq u$.
\end{proof}
\section{Regular languages for stable subgroups.}\label{sec:reg_language_for_stable_subgroups}

In this section, we prove Theorem \ref{intro_thm:stable_subgroups}, that stable subgroups of Morse local-to-global groups are characterized by the regularity of the language of Morse geodesic words that represent elements of the subgroup. At the end of  section, we will present examples that show this characterization fails to hold if the hypothesis of stability on the subgroup is weekend.

\begin{definition}[The language $L_H$]\label{defn:L_H}
Let $G$ be a group generated by the finite set $A$. For a subgroup $H \leq G$, let $L_H$ denote the language of all geodesic words $w$ in $A^\star$  so that $\overline{w} \in H$.
\end{definition}

\begin{theorem}[Characterizing stable subgroups]\label{thm:regular_language_characterizing_stable_subgroups}
Let $G$ be a Morse local-to-global group with finite generating set $A$. A subgroup $H \leq G$ is stable if  and only if there exists a Morse gauge $M$ so that $L_H$ is a regular language all of whose elements are $M$-Morse geodesic words in $\cay(G,A)$.
\end{theorem}

\begin{proof}
The sufficient condition will be shown in Corollary \ref{cor:regular+Morse_implies_stable} and the necessary condition will be proved in Corollary \ref{cor:regular language bijecting with stable subgroup}.
\end{proof}

The sufficient condition for Theorem \ref{thm:regular_language_characterizing_stable_subgroups} follows from a much more general lemma that says if $L$ is a regular language consisting of words that represent elements of a subgroup $H$, then the paths in the Cayley graph corresponding to the words in $L$ must stay uniformly close to $H$. In particular, this direction does not need the hypothesis that the ambient group $G$ has the Morse local-to-global property.

\begin{lemma} \label{lem:regular_languages_of_representative_implies_quasi-convex}
Let $G$ be a group generated by the finite set $A$, and let $H$ be a subgroup of $G$. Suppose $L$ is a regular language such that $\overline{w} \in H$ for each $w \in L$. If $k \geq 0$ is the number of vertices of a pruned finite state automaton that accepts $L$, then for each $w \in L$ the path in $\cay(G,A)$ from $e$ to $\overline{w}$ labeled by $w$ is contained in the $k$-neighborhood of $H$ in $\cay(G,A)$.
\end{lemma}

\begin{proof}
Let $\mc{G}$ be a pruned finite state automaton that accepts $L$ and let $k$ be the number of vertices of $\mc{G}$. Let $a_1\cdots a_n$ be the letters of a word $w \in L$.  Let $s_0,s_1,\dots, s_n$ be the sequence of vertices of $\mc{G}$ that appear along the directed path that reads $w$. Recall, this implies $s_n$ is an accept state of $\mc{G}$.  Since  $\mc{G}$ contains at most $k$ vertices, there exists a directed path with at most $k-1$ edges from $s_i$ to $s_n$ for each $i \in \{1,\dots,n\}$ (this path need not read a subword of $w$). Thus, for each $i \in \{1,\dots,n\}$ there exists a word $u_i \in A^\star$ so that $a_1\cdots a_i u_i$ is read by a path of $\mc{G}$ that ends at the accept state $s_n$ and $\ell(u_i) \leq k-1$. Hence $\overline{a_1\cdots a_i u_i} \in H$ since $a_1\cdots a_i u_i \in L$. Since $\ell(u_i) \leq k-1$, this implies the group element $\overline{a_1\cdots a_i}$ is contained in the $(k-1)$-neighborhood of $H$ in $\cay(G,A)$ for each $i \in \{1,\dots,n\}$. Thus, the path in $\cay(G,A)$ from $e$ to $\overline{w}$ labeled by $w$ is contained in the $k$-neighborhood of $H$ in $\cay(G,A)$.
\end{proof}

\begin{corollary}\label{cor:regular+Morse_implies_stable}
Let $G$ be a finitely generated group with finite generating set $A$. If the language $L_H$ from Definition \ref{defn:L_H} is regular and there exists a Morse gauge $M$ so that each word in $L_H$ is an $M$-Morse geodesic word in $\cay(G,A)$, then $H$ is stable.
\end{corollary}

\begin{proof}
By Lemma \ref{lem:regular_languages_of_representative_implies_quasi-convex}, if $L_H$ is regular, then there exists $k \geq 0$ so that  for every $h \in H$,  any geodesic from $e$ to $h$ in $\cay(G,A)$ is contained in the $k$-neighborhood of $H$ in $\cay(G,A)$. Thus, if there exists a Morse gauge $M$ so that each word in $L_H$ is also $M$-Morse, then $H$ is $(M,k)$-stable.
\end{proof}

For the necessary condition, we also prove a general result, this time about the regularity of languages whose corresponding paths in the Cayley graph stay uniformly close to a subgroup.

\begin{lemma}\label{lem:quasiconvex_subgroups_regular_language}
Let $G$ be a finitely generated group with a finite generating set $A.$ For any subgroup $H$ of $G$ and $k\geq 0$, the following languages are regular.
\begin{enumerate}
    \item The language consisting of all words $ w\in A^\star$ so that all the vertices of the path in  $\cay(G,A)$ starting at $e$ and labeled by $w$ are within the $k$-neighborhood of $H$. \label{item:close_paths_are_regular}
    \item The language $L_{H, k}$ of all words $ w\in A^\star$ so that $\overline{w} \in H$ and all the vertices of the path in $\cay(G,A)$ starting at $e$ and labeled by $w$ are within the $k$-neighborhood of $H$. \label{item:close_representatives_of_H_are_regular}
\end{enumerate}
\end{lemma}

\begin{proof}
Before defining the finite state automata that will read our languages, note the following observations. If $g$ is a vertex of $\cay(G,A)$ that is contained in the $k$-neighborhood of $H$, then there exist $h \in H$ so that $|g^{-1}h| \leq k$; see Figure \ref{fig:States are heights from H}. This suggests the vertices for both of our automata should be the set of all elements of $G$ contained in $B(e,k)$. Further, if for some $a \in A$, $g\overline{a}$ is also contained in the $k$-neighborhood of $H$, then there must exists $g' \in B(e,k)$ so that $g\overline{a}(g')^{-1} \in H$; see Figure \ref{fig:k_nieghborhood_of_H_edges}. This suggests that there should be an edge labeled by $a \in A$ in our automata from the vertex $g \in B(e,k)$ to the vertex $g' \in B(e,k)$ whenever $g \overline{a} (g')^{-1} \in H$.

\begin{figure}[ht]
    \centering

    \begin{tikzpicture}[scale=.3]
\draw[ thin ,orange] (-8,0) -- ++(18,0);

\draw[orange, thin] (-8,0) -- (2,6);

\draw[orange, thin] (10,0) -- (20,6);
\draw[ thin ,orange] (2,6) -- ++(18,0);

\draw[fill=black] (7,3) circle (0.2cm);

\node at (6,4) {$h$};

\draw[fill=black] (7,11) circle (0.2cm);

\draw[black, thick,-] (7,3)-- ++(0,8);

\node at (6,11.5) {$g$};
\node at (9,7) {$\leq k$};

\node[below, black] at (1,-.5) {$H$};

\end{tikzpicture}
\caption{Group element in the $k$-neighborhood of $H$ can be represented by an element $h \in H$ times an element in $B(e,k)$.}
\label{fig:States are heights from H}
\end{figure}
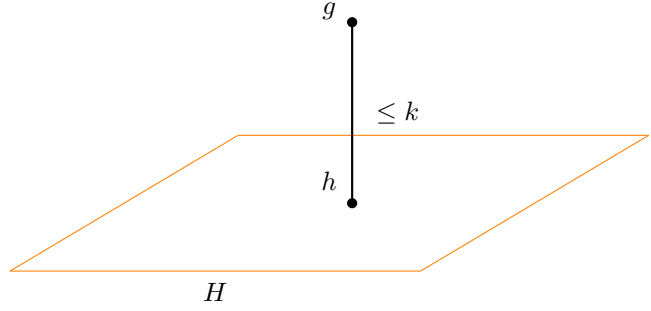

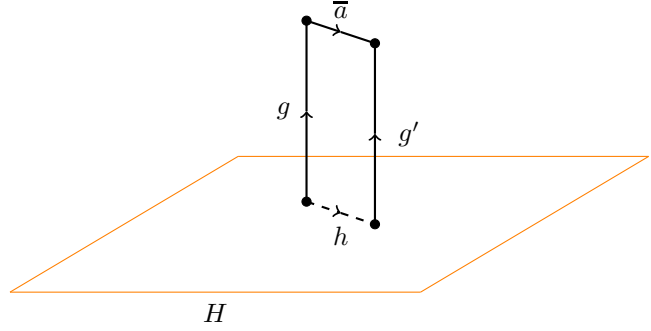
\begin{figure}[ht]
    \centering

    \begin{tikzpicture}[scale=.3]
\draw[ thin ,orange] (-8,0) -- ++(18,0);

\draw[orange, thin] (-8,0) -- (2,6);

\draw[orange, thin] (10,0) -- (20,6);
\draw[ thin ,orange] (2,6) -- ++(18,0);

\draw[fill=black] (5,4) circle (0.2cm);
\draw[fill=black] (5,12) circle (0.2cm);

\draw[fill=black] (8,3) circle (0.2cm);
\draw[fill=black] (8,11) circle (0.2cm);

\draw[black, thick, ->] (5,4)-- ++(0,4);
\draw[black, thick] (5,8) -- (5,12);
\draw[black, thick, ->] (8,3)-- ++(0,4);
\draw[black, thick] (8,7) -- ++(0,4);
\draw[black, thick, ->] (5,12) -- (6.5,11.5);
\draw[black, thick] (6.5,11.5) -- (8,11);

\draw[dashed, thick, ->] (5,4) -- (6.5,3.5);
\draw[dashed, thick] (6.5,3.5) -- (8,3);

\node at (4,8) {$g$};

\node at (6.5,12.5) {$\overline{a}$};
\node at (9.5,7) {$g'$};
\node at (6.5,2.5) {$h$};

\node[below, black] at (1,-.5) {$H$};

\end{tikzpicture}
\caption{If $g$ and $g\overline{a}$ are both in te $k$-neighborhood of $H$, then there exist $g' \in B(e,k)$ so that $g\overline{a}(g')^{-1}$ is an element $h \in H$.}
\label{fig:k_nieghborhood_of_H_edges}

\end{figure}

We now define the automata $\mc{G}_1$ and $\mc{G}_2$ in line with these observations. They will have the same set of vertices and edges and differ only in their set of accept states. Define $\mc{G}_1$ and $\mc{G}_2$ to be the following finite state automata with alphabet $A$.
\begin{itemize}
    \item The vertices for both $\mc{G}_1$ and $\mc{G}_2$ are the elements of $B(e,k)$.
    \item For both $\mc{G}_1$ and $\mc{G}_2$, there is an edge labeled by $a \in A$  from $g$ to $g'$ if $g \overline{a} (g')^{-1} \in H$. Since $A$ and $B(e,k)$ both contain only finitely many elements, there are only  a finite number of edges in both $\mc{G}_1$ and $\mc{G}_2$.
    \item The initial state for both $\mc{G}_1$ and $\mc{G}_2$ is the vertex $e$.
    \item All vertices of $\mc{G}_1$ are accept states for  $\mc{G}_1$. The only accept state for $\mc{G}_2$ is the vertex $e$.
\end{itemize}
We verify that $L_{H,k}$ is the language accepted by $\mc{G}_2$. The proof that the language accepted by $\mc{G}_1$ is the language described in (\ref{item:close_paths_are_regular}) is very similar.

Let $a_1\cdots a_n$ be the letters of a word $w \in L_{H,k}$. We verify that $w$ is read by a path in $\mc{G}_2$ that starts and ends at $e$. Let $g_0 =e$ and for each $i \in\{1,\dots,n\}$, let $g_i = \overline{a_1\cdots a_i} \in G$. Note, $g_n \in H$ because $a_1\cdots a_n = w \in L_{H,k}$.  Since each $g_i$ is contained in the $k$-neighborhood of $H$,  there exists $h_i \in H$ so that $h_i^{-1} g_i \in B(e,k)$ for each $i \in \{0,\dots, n\}$. Further, since $g_0$ and $g_n$ are elements of $H$, we choose $ h_0 = g_0$ and $h_n = g_n$. Let $s_i = h_i^{-1} g_i$. This ensures that each $s_i$ is a vertex of $\mc{G}_2$ and $s_n = e$. By  construction, $g_i \overline{a_{i+1}} g_{i+1}^{-1} = e$, thus we have \[s_i \overline{a_{i+1}} s_{i+1}^{-1} = h_{i}^{-1} g_i \overline{a_{i+1}} g_{i+1}^{-1} h_{i+1}  = h_{i}^{-1} h_{i+1} \in H\] for each $i \in \{0,\dots,n-1\}$. This implies there is a directed path in $\mc{G}_2$ that reads the word $a_1\cdots a_n$ and has vertices $e=s_0,\dots,s_n=e$. Hence $w = a_1\cdots a_n$ is read by a path in $\mc{G}_2$ that starts and ends at $e$.

Now, assume $w$ is a word in $A^\star$ that is read by a directed path $\eta$ in $\mc{G}_2$ that starts and terminates at the vertex $e$. We verify that $w \in L_{H,k}$. Let $a_1\cdots a_n$ be the letters of $w$ or equivalently the labels of the edges of the path $\eta$. For each $i \in \{0,\dots n \}$, let $s_i$ be the vertex of $\eta$ so that there is an edge labeled by $a_{i+1}$ from $s_i$  to $s_{i+1}$. Note, $s_0 = s_n = e$.   By construction of $\mc{G}_2$, there exists $h_i \in H$ so that $s_i \overline{a_{i+1}} (s_{i+1})^{-1} =h_i$ for each $ i \in\{0,\dots, n-1\}$.  Now, $\overline{a_1\cdots a_n} = \overline{w} \in H$ since $h_0\cdots h_{n-1} \in H$ and
 \[\overline{a_1\cdots a_n} = s_0\overline{a}_1(s_1)^{-1} s_1 \overline{a}_2 (s_2)^{-1} \cdots s_{n-1}\overline{a}_n (s_{n})^{-1} = h_0\cdots h_{n-1}.\tag{$\ast$} \label{eq:regular_implies_quasi-convex}\]
 
Since we have shown $\overline{w} \in H$, to prove $w \in L_{H,k}$, it remains to prove that the distance from $H$ to $\overline{a_1\dots a_i}$ is bounded by $k$ for all $i \in\{1,\dots n-1\}$. 
Applying a similar argument as in (\ref{eq:regular_implies_quasi-convex}) we have \[\overline{a_1 \cdots a_i} = h_0 \cdots h_{i-1} s_i \overline{a}_i\] for each $ i \in\{1,\dots,n\}$. 
Since $h_0\cdots h_i \in H$, the distance in $\cay(G,A)$ from $H$ to $\overline{a_1\cdots a_i}$ is at most the distance from $H$ to $s_i \overline{a}_i$. 
However, $s_i \overline{a_i} =  h_{i+1} s_{i+1}$ by construction. Thus, the distance from $H$ to $s_i \overline{a}_i$, and hence from $H$ to $\overline{a_1 \cdots a_i}$, is at most $|s_{i+1}|\leq k$ for all $i \in \{1,\dots, n-1\}$.
\end{proof}

When $H$ is a stable subgroup, there is a Morse gauge $M$ and a constant $k \geq 0$ so that every geodesic word starting at $e$ and ending in $H$ is $M$-Morse and contained in the $k$-neighborhood of $H$. 
As $M$-Morse geodesics are locally $M$-Morse, the language $L_H$ consists precisely of the  geodesics that are locally $M$-Morse that start at $e$ and end in $H$. In other words, $L_H$ is the intersection of the languages $L_{H,k}$ from Lemma \ref{lem:quasiconvex_subgroups_regular_language} and $L_M$ from Definition \ref{defn:L_M}. 
When the ambient group $G$ is a Morse local-to-global group, both $L_{H,k}$ and $L_M$ are regular implying $L_H$ is also regular.

\begin{corollary} \label{cor:regular language bijecting with stable subgroup}
Let $G$ be a Morse local-to-global group with finite generating set $A$. If $H$ is $(M,k)$-stable in $\cay(G,A)$, then:
\begin{enumerate}
    \item The language $L_H$ is a regular language where each element is an $M$-Morse geodesic word in $\cay(G,A)$.
    \item There is a regular language $J_H$ that bijects with $H$ and where each element is an $M$-Morse geodesic word in $\cay(G,A)$. \label{item:bijecting_with_stable_subgroup}
\end{enumerate}
\end{corollary}

\begin{proof}
Let $L_M$ be the language from Definition \ref{defn:L_M} and $L_{H,k}$ be the language from Lemma \ref{lem:quasiconvex_subgroups_regular_language}. As described in the paragraph before the corollary, every element of $L_H$ is an $M$-Morse geodesic word and $L_H$ is precisely $L_{H,k} \cap L_M$. The language $L_{H,k}$ is always regular (Lemma \ref{lem:quasiconvex_subgroups_regular_language}), while the language $L_M$ is regular because $G$ is a Morse local-to-global group (Corollary \ref{cor: Morse geodesic form regular langauge}). Since the intersection of two regular languages is regular by Lemma 1.4.1 of \cite{Word_Processing}, $L_H$ is also regular.

For the second claim, let $J_M$ be the sublanguage of $L_M$ provided by Corollary \ref{cor: lexographically least Morse geodesics}. As per Corollary \ref{cor: lexographically least Morse geodesics}, $J_M$ contains precisely one word for each element of $G$ represented by a word in $L_M$. Therefore $J_H = J_M \cap L_{H,k}$ is a regular  language that bijects with $H$ and where each element is an $M$-Morse geodesic word in $\cay(G,A)$.  Note, since the elements of $L_M$ are only locally $M$-Morse, the fact that the elements of $L_H$ or $J_H$ are $M$-Morse words is because $H$ is stable and not because $L_H,J_H \subseteq L_M$.
\end{proof}

Since Corollary \ref{cor:regular language bijecting with stable subgroup} produces a regular language that bijects with a stable subgroup, we can immediately conclude that stable subgroups of Morse local-to-global groups have rational growth.

\begin{corollary}
Let $H$ be a stable subgroup of a Morse local-to-global group $G$. For any finite generating set $A$ of $G$, if $f_{H,A}(n)=|B(e,n) \cap H|$, then there exist polynomials $P(x), Q(x) \in \mathbb{Q}[x]$ so that 
\[\sum_{n=0}^\infty f_{H,A}(n) \cdot x^n = \frac{P(x)}{Q(x)}. \]
\end{corollary}

\begin{proof}
Let $J_H$ be the regular language that bijects with $H$ as in part (\ref{item:bijecting_with_stable_subgroup}) of Corollary \ref{cor:regular language bijecting with stable subgroup}. Since $J_H$ consists of geodesic words, the number of words in $J_H$ with word  length $n$ is precisely $|B(e,n) \cap H| = f_{H,A}(n)$. Since $J_H$ is regular, Theorem \ref{thm: rational growth of languages} implies the desired statement.
\end{proof}

\subsection{The need for stability}\label{sec:need for stability}

Stable subgroups are known to be characterized as subgroups that are both hyperbolic and satisfy a similar convexity property as Morse quasi-geodesics. 

\begin{theorem}[{\cite[Theorem 4.8]{Tran_SQC}\cite[Lemma 3.8]{CordesDurham2017}\cite[Lemma 3.3]{DurhamTaylor2015}}]\label{thm:stable iff hyperbolic and Morse}
Let $G$ be a group with finite generating set $A$. A subgroup $H$ is stable if  and only if 
\begin{itemize}
    \item (Hyperbolicity) $H$ is a hyperbolic group;
    \item (Morseness) there is a Morse gauge $M$ so that every $(k,c)$-quasi-geodesic in $\cay(G,A)$ with endpoints in $H$ is contained in the $M(k,c)$-neighborhood of $H$ in $\cay(G,A)$.
\end{itemize}
\end{theorem}

In light of this, one might wonder if hyperbolicity or Morseness alone is sufficient for the conclusion of Theorem \ref{thm:regular_language_characterizing_stable_subgroups}. We now give two specific examples of subgroups that show that the answer is no; Theorem \ref{thm:regular_language_characterizing_stable_subgroups} fails to hold under the assumption of hyperbolicity or Morseness alone.

\begin{example} [Morse is not sufficient] Let $H$ be the group $\mathbb{Z}^2\rtimes\mathbb{Z}_2$ with the following presentations
\begin{center}

$H= \langle x,y,z, d, \alpha  \mid    \alpha ^2,xyx^{-1}y^{-1},\alpha x \alpha y^{-1}, xyd^{-1},zx^{-2}\rangle$.
\end{center}
 Example 4.4.2 in \cite{Word_Processing} shows that the language of all geodesic words in $H$ is not a regular language. Now, let $G=H \ast \mathbb{Z}$, given with the presentation 

\begin{center}
    $G= \langle x,y,z,w, d, \alpha  \mid    \alpha ^2,xyx^{-1}y^{-1},\alpha x \alpha y^{-1}, xyd^{-1},zx^{-2}\rangle$.
\end{center}

Since $H$ is a $\CAT(0)$ group and  $G$ is hyperbolic relative to $H$, we have that $G$ is a Morse local-to-global group by Theorem \ref{thm:Morse local-to-global groups} ($G$ is also a $\CAT(0)$ group itself). Further, as a peripheral subgroup of a relatively hyperbolic group, $H$ satisfies the Morseness condition of Theorem \ref{thm:stable iff hyperbolic and Morse}; see  \cite[Lemma 4.15]{DrutuSapir_Rel_hyp}. However, the language of geodesic words in $\cay(G,\{x,y,z,w,d,\alpha\})$ representing elements of $H$ is not regular since the language of geodesic words in $\cay(H,\{x,y,z,w,d\})$ is not regular.
\end{example}

\begin{example}[Hyperbolicity is not sufficient] \label{Hyperbolicity alone is not sufficient}

For this example, we will use the right-angled Artin group 
\begin{center}
    $G= \mathbb{Z}^2 \ast \mathbb{Z} =\langle a,b,c \mid [a,b]\rangle$,
\end{center}
which is Morse local-to-global by virtue of being $\CAT(0)$. Let $H$ be the subgroup $\langle ab \rangle$. For each $n \in \mathbb{N}$, the word $a^nb^n$ is an element of $L_H$. However, the path in $\cay(G,\{a,b,c\})$ starting at $e$ and labeled by $a^nb^n$ contains vertices, namely $a^n$, that are $n$-far from $H$. Hence, by Lemma \ref{lem:regular_languages_of_representative_implies_quasi-convex}, $L_H$ cannot be regular. 

Another source of counter-examples arises from distorted subgroups of hyperbolic groups. In a hyperbolic group, a subgroups is stable precisely when it is undistorted. Further, for any finite generating set, the language of geodesic words representing elements of a subgroup $H$ is regular if and only if $H$ is undistorted \cite[Theorem 2.2]{GerSho1991}. Thus, if $H$ and $G$ are both hyperbolic groups, but $H$ is a distorted subgroup of $G$, then the language $L_H$ is not regular despite $H$ being hyperbolic. An example where this occurs is a hyperbolic surface-by-cyclic group $G = \pi_1(S) \rtimes \mathbb{Z}$. The fiber subgroup $H = \pi_1(S)$ is both hyperbolic and distorted in $G$.
\end{example}

\section{The growth rate of a stable subgroup}\label{sec:growth_rate_gap}
We now apply Corollary \ref{cor:regular language bijecting with stable subgroup} to prove that, under certain hypotheses, the growth rate of a stable subgroup of a Morse local-to-global group is strictly less than the growth rate of the entire group. This proves Theorems  \ref{intro_thm:growth_torsion_free},  \ref{intro_thm:growth_res_finite}, and \ref{intro_cor:MCG_case_of_Growth} from the introduction.

\begin{theorem}\label{thm:definite_gap_between_growths}
Let $G$ be a Morse local-to-global group and let $H$ be an infinite index stable subgroup of $G$. If either $G$ is virtually torsion-free or $H$ is residually finite, then for any finite generating set $A$ of $G$, we have $\lambda_{H,A}< \lambda_{G,A}.$
\end{theorem}

Both cases of Theorem \ref{thm:definite_gap_between_growths} reduce to the following technical result, the proof of which closely follows Dahmani, Futer, and Wise's proof in the case of quasi-convex subgroups of torsion-free hyperbolic groups \cite[Theorem 5.1]{DahmaniFuterWise2018}.

\begin{proposition}\label{prop:growth_gap_technical}
Let $G$ be a Morse local-to-global group and let $H$ be an non-trivial infinite index stable subgroup of $G$.  If there exists $H' \leq H$ and $g \in G - \{e\}$ so that $H'$ is finite index in $H$ and $\langle H' , g \rangle \cong H' \ast \langle g\rangle$, then for any finite generating set $A$ of $G$, we have $\lambda_{H,A} < \lambda_{G,A}$.
\end{proposition}

\begin{proof}
Let $A$ be a finite generating set for the Morse local-to-global group $G$. Let $H$ be an infinite index stable subgroup of $G$ and assume there exists $H' \leq H$ and $g \in G -\{e\}$ so that $H'$ is finite index in $H$ and $\langle H' , g \rangle \cong H' \ast \langle g\rangle$.

Since $H'$ is finite index in $H$, $H'$ is also a stable subgroup of $G$. Thus, there exists a pruned finite state automaton $\mc{G}$ accepting the regular language $J=J_{H'}$ from Corollary \ref{cor:regular language bijecting with stable subgroup}.  Since  Corollary \ref{cor:regular language bijecting with stable subgroup} shows that $J$ bijects with $H'$, Corollary \ref{Sequence of equalities} implies that $\lambda_{H',A}=\rho_\mc{G}$.

Let $w$ be a geodesic word in $A^\star$ that represents the element $g$ for which $\langle  H' ,g \rangle \cong  H' \ast \langle g \rangle $. Let $s_0$ be the unique initial state of $\mc{G}$. 
From $\mc{G}$, we construct a new finite state automaton $\mc{G}'$ as follows:
\begin{itemize}
    \item The states of $\mc{G}'$ are precisely the states of $\mc{G}$. The initial state of $\mc{G}'$ is $s_0$, which is the initial state of $\mc{G}$.
    \item The accept states of $\mc{G}'$ are the accept states of $\mc{G}$ plus $s_0$.
    \item If there is a directed edge of $\mc{G}$ from the state $s$ to the state $t$, then there is also a directed edge of $\mc{G}'$ from $s$ to $t$. 
    \item  For each accept state $s$ of $\mc{G}$, $\mc{G}'$ has an additional directed edge with label $w$ starting at $s$ and ending at $s_0$. Note, this includes a directed edge starting and ending at $s_0$.

\end{itemize}

Let $J'$ be the language accepted by the new automaton $\mc{G}'$. Since $\mc{G}'$ is obtained from $\mc{G}$ by adding paths between accept states of $\mc{G}'$, $\mc{G}$ being pruned implies $\mc{G}'$ is pruned. Thus, $\rho_\mc{G} = \lambda_{J}$ and $\rho_{\mc{G}'} = \lambda_{J'}$. 
By construction, there is also a directed path from every accept state in $\mc{G}'$ to $s_0$. Thus, for every pair of  vertices $s,t$ of $\mc{G}'$ there exists a directed path from $s$ to $t$ and from $t$ to $s$. In particular, Lemma \ref{lem:proper subgraphs have smaller eigenvalue} implies  \[\lambda_{J}=\rho_{\mc{G}} < \rho_{\mc{G}'} = \lambda_{J'}.\] 

Since $\langle  H' ,g \rangle \cong  H' \ast \langle g \rangle $, the map  $J' \rightarrow G$ given by $u \to \overline{u}$ is injective and hence $J'$ is a regular language with alphabet $A$ that bijects with the subgroup $\langle H',g\rangle $. Therefore, $\lambda_{J'} = \lambda_{\langle  H' ,g \rangle,A}$ by Corollary \ref{Sequence of equalities}, and we have \[\lambda_{H',A} =\lambda_{J}= \rho_\mc{G} < \rho_{\mc{G}'} =\lambda_{J'} = \lambda_{\langle  H' ,g \rangle, A}.\]
 
By definition, $\lambda_{\langle  H' ,g \rangle, A} \leq \lambda_{G,A}$, thus we have $\lambda_{H',A} < \lambda_{G,A}$. The following  claim  verifies $\lambda_{H,A} = \lambda_{H',A}$, completing the proof that $\lambda_{H,A} < \lambda_{G,A}$.
\end{proof}

\begin{claim}\label{lem:finite index implies same growth}
Let $G$ be a group generated by the finite set $A$. If $H$ is a subgroup of $G$ and $H'$ is a finite index subgroup of $H$, then $\lambda_{H,A} = \lambda_{H',A}$.
\end{claim}

\begin{proof}
Let $k$ be the index of $H'$ in $H$.  Let $h_1H',\dots,h_kH'$ be the $k$ distinct cosets of $H'$ and let $K = \max\{|h_1|,\dots, |h_k|\}$. Since the cosets of $H'$ partition $H$, we have \[|B(e,n) \cap H| = \sum_{i=1}^k | B(e,n) \cap h_i H'| = \sum_{i=1}^k | B(h_i^{-1},n) \cap H'| \leq k|B(e,n+K) \cap H'|.\] This implies $\lambda_{H,A} = \lambda_{H',A}$ since $K$ and $k$ do not depend on $n$.
\end{proof}

To apply Proposition \ref{prop:growth_gap_technical} to the situations in Theorem \ref{thm:definite_gap_between_growths}, we employ the following results of Russell, Spriano, and Tran to find the required $H'$ and $g \in G$. Both of these results are consequence of a general combination theorem for stable subgroups of Morse local-to-global groups \cite{Morse-local-to-global}.

\begin{theorem}[{\cite[Corollary 3.6]{Morse-local-to-global}}]\label{thm: free factor} Let $G$ be a torsion-free, Morse local-to-global group. If $H$ is a non-trivial, infinite
index, stable subgroup of $G$, then there is an infinite order element $g \in G$ such that $\langle H, g \rangle = H \ast \langle g \rangle $ and $\langle H,g \rangle$ is stable in $G$. 
\end{theorem}

\begin{theorem}[{Special case of \cite[Corollary 3.4]{Morse-local-to-global}}]\label{thm:res_finit_combination}
Let $G$ be a Morse local-to-global group. If $H$ is a residually finite, stable subgroup and $g \in G$ is a Morse element so that $H \cap \langle g\rangle = \{e\}$, then there is a finite index subgroup $H' \leq H$ and $n \in \mathbb{N}$ so that $\langle H', g^n \rangle \cong H' \ast \langle g^n \rangle$ and  $\langle H', g^n \rangle$ is stable in $G$.
\end{theorem}

\begin{proof}[Proof of Theorem \ref{thm:definite_gap_between_growths}]
Let $G$ be a Morse local-to-global group and $H <G$ be an infinite index stable subgroup. Fix a finite generating set $A$ for $G$. We can assume $H$ is not finite  as $\lambda_{H,A} < \lambda_{G,A}$ whenever $H$ is finite and $G$ is infinite.

First assume $G$ is virtually torsion-free and let $G'$ be a  finite index torsion-free subgroup of $G$. Since the Morse local-to-global property is invariant under quasi-isometry, $G'$ is also a Morse local-to-global group. Define $H'$ to be $G' \cap H$. Because $H'$ is a finite index subgroup of $H$, $H'$ is a stable subgroup of both $G$ and $G'$. Since $G'$ is torsion-free, Theorem \ref{thm: free factor} provides a non-trivial $g \in G' < G$, so that $\langle H', g \rangle \cong H' \ast \langle g\rangle$. Thus, the requirements for Proposition \ref{prop:growth_gap_technical} are satisfied implying $\lambda_{H,A} < \lambda_{G,A}$.

Now assume $H$ is residually finite. Let $h$ be an infinite order element of $H$.  Since $H$ is stable, $h$ is a Morse element of $G$. Because all infinite index stable subgroups have finite width \cite{AMST_intersection_of_stable_subgroups,Tran_SQC}, there exists $q \in G$ so that $H \cap qHq^{-1}$ is finite. In particular, $H \cap \langle qhq^{-1} \rangle = \{e\}$, and $qhq^{-1}$ is a Morse element because $h$ is a Morse element. Thus, Theorem \ref{thm:res_finit_combination} provides the subgroup $H' \leq H$ and the element $g \in G$ required by Proposition \ref{prop:growth_gap_technical}  to ensure $\lambda_{H,A} < \lambda_{G,A}$.
\end{proof}

\section{Applications to the Morse boundary}\label{sec:boundary_applications}
We conclude with some applications of our previous results to the dynamics of a Morse local-to-global group on its Morse boundary.

The Morse boundary of a proper geodesic metric space $X$ is defined as such: consider the set of all Morse geodesic rays in $X$ based at a point $e$ and modulo asymptotic equivalence.  This collection is called the \emph{Morse boundary} of $X$ and is denoted by $\partial_{*} X$ or $\partial_{*} X_e$ if we wish to emphasize the basepoint $e$.

To topologize the boundary, first fix a Morse gauge $M$ and consider the subset of the Morse boundary that consists of all rays in $X$ with Morse gauge at most $M$:  \begin{equation*} \partial_{*}^M X_e= \{[\alpha] \mid \exists \beta \in [\alpha] \text{ that is an $M$-Morse geodesic ray with base point }e\}. \end{equation*} We topologize this set with the compact-open topology. This topology is equivalent to the one defined by the following system of neighborhoods, $\{V_n(\alpha) \mid n \in \mathbb{N} \}$, at each point $[\alpha] \in \partial_{*}^M X_e$: the set $V_n( \alpha)$ is the set of $M$-Morse geodesic rays $\gamma \colon [0,\infty) \to X$ with $\gamma(0)=e$ and $d(\alpha(t), \gamma(t))< \delta_M$ for all $t<n$, where $\delta_M$ is a constant that depends only on $M$. 

Let $\mathcal M$ be the set of all Morse gauges. We put a partial ordering on $\mathcal M$ so that  for two Morse gauges $M, M' \in \mathcal M$, we say $M \leq M'$ if and only if $M(\lambda,\epsilon) \leq M'(\lambda,\epsilon)$ for all $\lambda \geq 1$ and $\epsilon\geq 0$. We define the Morse boundary of $X$ to be
 \begin{equation*} \partial_{*} X_e=\varinjlim_\mathcal{M} \partial^M_* X_e \end{equation*} with the induced direct limit topology, i.e., a set $U$ is open in $\partial_{*} X_e$ if and only if $U \cap \partial^M_* X_e$ is open for all $M \in \mc{M}$.  For more details on the Morse boundary see \cite{Cordes2017}. 

The first salient feature of the Morse boundary is that it is invariant under quasi-isometry regardless of the geometry of the space. 

\begin{theorem}[\cite{Cordes2017}]\label{thm:Morse_boundar_QI_invariant}
Let $X$ and $Y$ be proper geodesic metric spaces. Every quasi-isometry $f\colon X \to Y$ induces a homeomorphism $\partial f \colon \partial_{*} X_e \to \partial_{*} Y_{f(e)}$.
\end{theorem}

For a finitely generated group $G$, we define the \emph{Morse boundary of $G$}, denoted $\partial_{*} G$, to be the Morse boundary of a Cayley graph of $G$ with respect to some finite generating set. The quasi-isometry invariance of the Morse boundary implies this is well defined up to homeomorphism. 

The action of $G$ on its Cayley graph induces an action by homeomorphisms on its Morse boundary. Cordes and Durham defined the following notation of the limit set of a subgroup in the Morse boundary. In the sequel, the notation $[e,g]$ for $g \in G$ denotes a geodesic from the identity $e$ to the element $g$ in the Cayley graph of $G$ with respect to some finite generating set. For a sequence $\{g_n\}$ in $G$ and $x \in \partial_{*} G$, we write $[e,g_n]\to x$ if $[e,g_n]$ converges uniformly on compact subsets to a geodesic ray representing $x$.

\begin{definition}[Limit set of a subgroup; {\cite[Definition 3.2]{CordesDurham2017}}]
Let $H$ be a subgroup of a finitely generated group $G$. The \emph{limit set} of $H$ in $\partial_{*} G$, denoted  $\Lambda_{*}(H),$ is defined to be $$\Lambda_{*}(H):=\{x \in \partial_{*} G \mid  \exists M \in \mc{M} \text{ and } h_n \in H \text{ so that } [e,h_n] \rightarrow x \text{ and each }  [e,h_n] \text{ is }M\text{-Morse} \}.$$
\end{definition}

If $g$ is a Morse element of $G$, then the cyclic subgroup $\langle g \rangle$ has a pair of limit points denoted $g^+$  and $g^-$; $g^+$ is the limit of the geodesics $[e,g^n]$, while $g^-$ is the limit of the geodesics $[e,g^{-n}]$. The Morse element $g$ acts with a form of north-south dynamics on the Morse boundary with $g^+$ being the attracting fixed point and $g^-$ being the repelling fixed point.

\begin{lemma}[North-south dynamics; {\cite[Corollary 6.8]{Qing19}}] \label{lem: North-south dynamics} Let $G$ be a finitely generated group. For any Morse element $g \in G$, if $U \subseteq \partial_{*} G$ is an open neighborhood of $g^{+}$ and $K \subseteq \partial_{*} G$ is a compact set with $g^{-} \notin K$, then there exists $m \in \mathbb{N}$ such that $g^{m}K \subseteq U$.
\end{lemma}

Regardless of the subgroup, limit sets are always closed in the Morse boundary.

\begin{lemma}[Limit sets are closed] \label{lem: limit set is closed }
Let $H$ be a subgroup of a finitely generated group $G$, and let $\{x_n\}$ be a sequence in  $\Lambda_{*}(H)$. If there exists $x \in \partial_{*} G$ such that $x_n \rightarrow x$ in the direct limit topology on $\partial_{*} G$, then $x \in \Lambda_{*} (H).$
\end{lemma}

\begin{proof}
Fix a finite generating set $A$ for $G$. Let $X=\cay(G,A)$ and let $\partial_{*} X_e$ be the Morse boundary of $X$ based at the identity $e$ of $G$. Let $\{x_n\} \subseteq  \Lambda_{*} (H)$ be a sequence that converges to $x$ in $\partial_{*} X_e$. By Lemma 5.3 of \cite{Qing19}, there exists a Morse gauge $M$ such that $x_n,x \in   \partial_{*}^{M} X_e$ for all $n \in \mathbb{N}$.
Let $\alpha, \alpha_n:[0, \infty) \rightarrow X$ be geodesic rays starting at $e$ with $[\alpha]=x$ and $[\alpha_n]=x_n.$ Since $x_n \rightarrow x$, we can pass to a subsequence so that $\{\alpha_n\}$ converges uniformly on compact subsets to $\alpha$. In particular, for each $i \in \mathbb{N}$, there exists $n_i \in \mathbb{N}$ so that for all $t \leq i$ we have $d(\alpha(t),\alpha_{n_i}(t)) \leq \delta_M$, where $\delta_M$ is a fixed constant depending only on $M$.
For each $i \in \mathbb{N}$, the fact that $[\alpha_{n_i}] = x_n \in \Lambda_{*}(H)$  implies there exists a sequence $\{h_{j}\}$ of elements of $H$ such that the sequence of geodesics $\{[e,h_j]\}$  converges uniformly on compact subsets to $\alpha_{n_i}$. Thus, for each $ i \in \mathbb{N}$, there exists a finite geodesic $\beta_i\colon [0,a_i] \to \cay(G,A)$  so that $a_i \geq i$, $\beta_i(0) = e$, $\beta_i(a_i) \in H$, and $d(\alpha_{n_i}(t),\beta_i(t))\leq 1$ for all $t \leq i$. Therefore, for each $i \in \mathbb{N}$ and $t<i$, we have \[d(\alpha(t),\beta_i(t))\leq d(\alpha(t),\alpha_{n_i}(t)) + d(\alpha_{n_i}(t),\beta_i(t) \leq \delta_M+1.\] Thus, $\{\beta_i\}$ converges uniformly on compact subsets to $\alpha$. Since each $\beta_i$ is a geodesic from the identity to an element of $H$, this implies $[\alpha] = x \in \Lambda_{*}(H)$
\end{proof}

Our first dynamical result contains our connection between regular languages and the Morse boundary.  We use the fact that the language of geodesics words that are $(B_M;M)$-locally  Morse is regular (Theorem \ref{thm:regular_language_for_Morse_geodesics}) to prove that the set of limit points of Morse elements in the Morse boundary is dense.

\begin{theorem}\label{thm: algebraic density of Morse boundary}  Let $G$ be a Morse local-to-global group. The set  $$Q_{*}(G):=\{g^{+} \mid  g \text{ is Morse}\}$$ is dense in the $\partial_{*} G$.
\end{theorem}

\begin{proof}
Fix a finite generating set $A$ and let $X= \cay(G,A)$. We wish to show that for every $x \in \partial_{*} X_e$, there exists a sequence of Morse elements $g_n \in G$ with $[g^{+}_n] \rightarrow x$.

Let $\alpha:[0, \infty) \rightarrow X$ be an $M$-Morse geodesic ray with $\alpha(0)=e$ and $x=[\alpha]$. Let $\mc{G}$  be the finite state automaton  from Corollary \ref{cor: Morse geodesic form regular langauge} accepting the language $L_M$ of geodesic words in $A^\star$ that are $(B_M;M)$-locally Morse. Let $m$ be the number of states of $\mc{G}$. We wish to show that for any $i \in \mathbb{N}$ there exists $g_i \in G$ such that the sequence of geodesics  $\{[e,g_i^n]\}^\infty_{n=1}$ fellows travels $\alpha$ for at least $i$-time.

For a word $w\in L_M$, let $s(w)$ be the state of $\mc{G}$ obtained by reading the word $w$. For each $i \in \mathbb{N}$, let $w_i$ be the geodesic word labeling the geodesic segment $\alpha\vert_{[0,i]}$. Since the total number of states of $\mc{G}$ is $m$, for each $i \in \mathbb{N}$, we must have $s(w_j)= s(w_k)$ for two distinct $j,k \in \{i,\dots, i+m+1\}$.   This implies that $w_{i+m+1} = u_i v_i q_i$ where $|u_i|\geq i$ and $s(u_i) = s(u_iv_i)$.   For each $n$, the word  $u_i v_i^n$ is therefore accepted by $\mc{G}$. Hence  $u_iv_i^n$ is a geodesic word that is $(B_M;M)$-locally Morse. Further, the infinite ray $\beta_i \colon [0,\infty) \to \cay(G,A)$ stating at $e$ and labeled by $u_iv_iv_iv_i\cdots$ is also a geodesic that is $(B_M;M)$-locally Morse.   By applying the Morse local-to-global property, there exists $M'$ depending only on $M$ so that $\beta_i$ is an $M'$-Morse geodesic ray and $u_iv_i^n$ is an $M'$-Morse geodesic word for all $i,n\in \mathbb{N}$; see Remark \ref{rem:Morse gauge for fellow traveling}. 

Let $g_i = \overline{u}_i\overline{v}_i\overline{u}_i^{-1}$. Since $\beta_i$ is an $M'$-Morse geodesic and $|\overline{u}_i|$ is  bounded, $g_i$ is a Morse element and $g_i^+ = [\beta_i]$ for all $i \in \mathbb{N}$. By construction, $d(\alpha(t),\beta_i(t))=0$ for all $t \leq i$.  Hence $[\beta_i] = g_i^+ \rightarrow x$ in $\partial_{*} X_e$, because  $\beta_i \rightarrow \alpha$ uniformly on compact subsets.
\end{proof}

\begin{remark}\label{rem:infinite_order_element}
The proof of Theorem \ref{thm: algebraic density of Morse boundary} ensures that any Morse local-to-global group with non-empty Morse boundary contains an infinite order Morse element. While this fact was already shown in \cite{Morse-local-to-global}, we note that the proof of Theorem \ref{thm: algebraic density of Morse boundary} does not directly use the Morse local-to-global property, but only the conclusion of Theorem  \ref{intro_thm:regular_language_for_Morse_Geodesics} that the set of $M$-Morse geodesic words is contained in a regular language all of whose elements are $M'$-Morse geodesic words. In particular, the proof of Theorem \ref{thm: algebraic density of Morse boundary} implies Theorem \ref{intro_thm:regular_language_for_Morse_Geodesics} cannot hold for Fink's examples of infinite torsion groups with non-empty Morse boundary \cite{Fink}.
\end{remark}

Our second dynamical result applies Theorem \ref{thm: algebraic density of Morse boundary} to conclude that the limit set of any infinite normal subgroup is the entire Morse boundary.

 \begin{theorem}\label{thm: normal subgroups fill}
If $H$ is an infinite normal subgroup of a Morse local-to-global group $G$, then the limit set of $H$ is the full Morse boundary of $G$. In notation, we have $\Lambda_{*}(H)=\partial_{*} G  $.

\end{theorem}

\begin{proof}[Proof of Theorem \ref{thm: normal subgroups fill}] We argue similarly to Theorem 3.1 of \cite{BowersRuane1996}. Let $G$ be a Morse local-to-global group and $H$ be a infinite normal subgroups of $G$. Let $X = \cay(G,A)$ where $A$ is some fixed finite generating set for $G$. We wish to show that for every $x \in \partial_{*} X_e$, we have $x \in \Lambda_{*}(H).$ 

By Theorem \ref{thm: algebraic density of Morse boundary}, for every open neighborhood $U$ of $x \in \partial_{*}G$, there exists some Morse element $g \in G$, such that $g^{+} \in U.$ By Corollary 3.6 of \cite{Morse-local-to-global},  there is $h\in H$ such that $h$ is a Morse element of $G$. Using Lemma \ref{lem: North-south dynamics}, there exists some $m \in \mathbb{N}$ such that $g^{m}h^{+} \in U$. The sequence $\{[e,g^{m}h^{i}g^{-m}]\}_{i=1}^\infty$ in $\cay(G,A)$ converges to $g^m h^{+} \in \partial_{*} X_e$. Since $H$ is normal, we have $g^{m}h^{i}g^{-m} \in H$. Hence, $ g^{m}h^{+} \in \Lambda_{*}(H).$ Since every open neighborhood of $x$ contains an element of the closed set $\Lambda_{*} (H)$ (Lemma \ref{lem: limit set is closed }), we conclude that $x \in \Lambda_{*} (H).$
\end{proof}

We conclude by proving a nice corollary to Theorem \ref{thm: normal subgroups fill} that shows a hyperbolic normal subgroup of a non-hyperbolic Morse local-to-global group cannot admit a Cannon--Thurston map. In order to formally define Cannon--Thurston maps for Morse boundaries, we must first make a simple observation: if $H=G$, then $\Lambda_*(H)=\partial_*H$.

\begin{definition}[Cannon--Thurston map for Morse boundaries]
Let $H$ be a  finitely generated subgroup of a finitely generated group $G$. There is a \emph{Cannon--Thurston map} from $\partial_{*} H$ to $\partial_{*} G$ if there exists a continuous function $\partial_\iota \colon \partial_* H \to \partial_*G$ with the property that for each $x \in \partial_{*} H$, if $[e,h_n]_H \to x$ for some sequence $\{h_n\}$ of elements of $H$, then $[e,h_n]_G \to \partial_\iota(x)$. Note, this definition ensures that if a Cannon--Thurston map exists, then its image is exactly $\Lambda_{*}(H)$.
\end{definition}

\begin{corollary}[Non-existence of Cannon--Thurston maps]\label{cor:no_cannon-thurston}
Let $G$ be a Morse-local-to-global group that is not hyperbolic and let $H$ be a normal hyperbolic subgroup of $G$, then a Cannon--Thurston map does not exist.
\end{corollary}

\begin{proof}
If $\partial_* G$ is empty, then no map exists, and thus a Cannon--Thurston map does not exist. So assume that $\partial_* G$ is not empty and that there is a Cannon--Thurston map $\partial_\iota \colon \partial_* H \to \partial_*G$. 

Since $H$ is hyperbolic, the Morse boundary of $H$, denoted $\partial_* H$, is simply the Gromov boundary of $H$ and is therefore compact \cite{Cordes2017}. Since $G$ is not hyperbolic, then $\partial_*G$ is not compact \cite{CordesDurham2017}. By Theorem \ref{thm: normal subgroups fill}, we know that $\partial_\iota$ is a surjective map, and it is continuous by definition of Cannon--Thurston map. Since the continuous image of a compact set is compact, we reach a contradiction. 
\end{proof}

\bibliography{bibliography}{}

\newcommand{\etalchar}[1]{$^{#1}$}
\begin{thebibliography}{BCG{\etalchar{+}}18}

\bibitem[ACT15]{ACT_growth_tightness}
Goulnara~N. Arzhantseva, Christopher~H. Cashen, and Jing Tao.
\newblock Growth tight actions.
\newblock {\em Pacific J. Math.}, 278(1):1--49, 2015.

\bibitem[AMST19]{AMST_intersection_of_stable_subgroups}
Yago Antol\'{\i}n, Mahan Mj, Alessandro Sisto, and Samuel~J. Taylor.
\newblock Intersection properties of stable subgroups and bounded cohomology.
\newblock {\em Indiana Univ. Math. J.}, 68(1):179--199, 2019.

\bibitem[BCG{\etalchar{+}}18]{BCGGS18}
Benjamin Beeker, Matthew Cordes, Giles Gardam, Radhika Gupta, and Emily Stark.
\newblock Cannon--{T}hurston maps for {CAT}(0) groups with isolated flats,
  2018.
\newblock arXiv:1810.13285.

\bibitem[Beh06]{Behrstock_pA_are_Morse}
Jason~A. Behrstock.
\newblock Asymptotic geometry of the mapping class group and {T}eichm\"{u}ller
  space.
\newblock {\em Geom. Topol.}, 10:1523--1578, 2006.

\bibitem[BGSS91]{BaumslagGerstenShapiroShort1991}
Gilbert Baumslag, Stephen~M. Gersten, Michael Shapiro, and Hamish Short.
\newblock Automatic groups and amalgams.
\newblock {\em Journal of Pure and Applied Algebra}, 76(3):229--316, 1991.

\bibitem[BH09]{BH}
Martin~R. Bridson and André Häfliger.
\newblock {\em Metric Spaces of Non-Positive Curvature}.
\newblock Springer, 2009.

\bibitem[BR96]{BowersRuane1996}
Philip~L. Bowers and Kim Ruane.
\newblock {\em Proceedings of the American Mathematical Society},
  124(04):1311--1314, 1996.

\bibitem[Can84]{Cannon1984}
James~W. Cannon.
\newblock The combinatorial structure of cocompact discrete hyperbolic groups.
\newblock {\em Geometriae Dedicata}, 16(2), 1984.

\bibitem[CD17]{CordesDurham2017}
Matthew Cordes and Matthew~G. Durham.
\newblock Boundary convex cocompactness and stability of subgroups of finitely
  generated groups.
\newblock {\em International Mathematics Research Notices}, 2019(6):1699--1724,
  2017.

\bibitem[CM19]{Cashen_Mckay_boundary}
Christopher~H. Cashen and John~M. Mackay.
\newblock A metrizable topology on the contracting boundary of a group.
\newblock {\em Trans. Amer. Math. Soc.}, 372(3):1555--1600, 2019.

\bibitem[Cor17]{Cordes2017}
Matthew Cordes.
\newblock Morse boundaries of proper geodesic metric spaces.
\newblock {\em Groups, Geometry, and Dynamics}, 11(4):1281--1306, 2017.

\bibitem[CS14]{ChSu2014}
Ruth Charney and Harold Sultan.
\newblock Contracting boundaries of {CAT}(0) spaces.
\newblock {\em Journal of Topology}, 8(1):93--117, 2014.

\bibitem[DFW18]{DahmaniFuterWise2018}
Fran{\c{c}}ois Dahmani, David Futer, and Daniel~T. Wise.
\newblock Growth of quasiconvex subgroups.
\newblock {\em Mathematical Proceedings of the Cambridge Philosophical
  Society}, 167(3):505--530, 2018.

\bibitem[DMS10]{DMS_divergence}
Cornelia Dru\c{t}u, Shahar Mozes, and Mark Sapir.
\newblock Divergence in lattices in semisimple {L}ie groups and graphs of
  groups.
\newblock {\em Trans. Amer. Math. Soc.}, 362(5):2451--2505, 2010.

\bibitem[DS05]{DrutuSapir_Rel_hyp}
Cornelia Dru\c{t}u and Mark Sapir.
\newblock Tree-graded spaces and asymptotic cones of groups.
\newblock {\em Topology}, 44(5):959--1058, 2005.
\newblock With an appendix by Denis Osin and Mark Sapir.

\bibitem[DT15]{DurhamTaylor2015}
Matthew~G. Durham and Samuel~J. Taylor.
\newblock Convex cocompactness and stability in mapping class groups.
\newblock {\em Algebraic {\&} Geometric Topology}, 15(5):2839--2859, 2015.

\bibitem[Eps92]{Word_Processing}
David B.~A. Epstein.
\newblock {\em Word Processing in Groups}.
\newblock A K Peters/CRC Press, 1992.

\bibitem[EZ18]{EikeZalloum}
Joshua Eike and Abdalrazzaq Zalloum.
\newblock Regular languages for contracting geodesics, 2018.
\newblock arXiv:1809.02692.

\bibitem[Far06]{FarbProblems}
Benson Farb, editor.
\newblock {\em Problems on mapping class groups and related topics}, volume~74
  of {\em Proceedings of Symposia in Pure Mathematics}.
\newblock American Mathematical Society, Providence, RI, 2006.

\bibitem[Fie20]{Elizabeth}
Elizabeth Field.
\newblock Trees, dendrites and the {C}annon-{T}hurston map.
\newblock {\em Algebr. Geom. Topol.}, 20(6):3083--3126, 2020.

\bibitem[Fin17]{Fink}
Elisabeth Fink.
\newblock Morse geodesics in torsion groups, 2017.
\newblock arXiv:1710.11191.

\bibitem[FM12]{primer}
Benson Farb and Dan Margalit.
\newblock {\em A primer on mapping class groups}, volume~49 of {\em Princeton
  Mathematical Series}.
\newblock Princeton University Press, Princeton, NJ, 2012.

\bibitem[Gek14]{Gekhtmen_convex_cocompact}
Ilya Gekhtman.
\newblock {\em Dynamics of convex cocompact subgroups of mapping class groups}.
\newblock ProQuest LLC, Ann Arbor, MI, 2014.
\newblock Thesis (Ph.D.)--The University of Chicago.

\bibitem[Gro87]{Gromov}
Mikhael Gromov.
\newblock Hyperbolic groups.
\newblock In {\em Essays in group theory}, volume~8 of {\em Math. Sci. Res.
  Inst. Publ.}, pages 75--263. Springer, New York, 1987.

\bibitem[GS91]{GerSho1991}
Stephen~M. Gersten and Hamish~B. Short.
\newblock Rational subgroups of biautomatic groups.
\newblock {\em The Annals of Mathematics}, 134(1):125, 1991.

\bibitem[HR03]{HoltRees2003}
Derek Holt and Sarah Rees.
\newblock {Regularity} {of} {quasigeodesics} {in} a {hyperbolic} {group}.
\newblock {\em International Journal of Algebra and Computation},
  13(05):585--596, 2003.

\bibitem[Kap96]{Ilya}
Ilya Kapovich.
\newblock {\em Detecting quasiconvexity: algorithmic aspects}, volume~25 of
  {\em DIMACS Ser. Discrete Math. Theoret. Comput. Sci.}, pages 91--99.
\newblock Amer. Math. Soc., Providence, R.I., 1996.

\bibitem[Kim20]{Kim2020}
Heejoung Kim.
\newblock Algorithms detecting stability and morseness for finitely generated
  groups.
\newblock {\em Journal of Algebra}, 554:106--138, 2020.

\bibitem[KMT17]{KoberdaMangahasTaylor2017}
Thomas Koberda, Johanna Mangahas, and Samuel~J. Taylor.
\newblock The geometry of purely loxodromic subgroups of right-angled artin
  groups.
\newblock {\em Transactions of the American Mathematical Society},
  369(11):8179--8208, 2017.

\bibitem[Liu19]{Qing19}
Qing Liu.
\newblock Dynamics on the morse boundary, 2019.
\newblock arXiv:1905.01404.

\bibitem[LW20]{LiWise2020}
Jiakai Li and Daniel~T. Wise.
\newblock No growth-gaps for special cube complexes.
\newblock {\em Groups, Geometry, and Dynamics}, 14(1):117--135, 2020.

\bibitem[Mit98]{Mitra1998}
Mahan Mitra.
\newblock Cannon--{T}hurston maps for hyperbolic group extensions.
\newblock {\em Topology}, 37(3):527--538, 1998.

\bibitem[Mur15]{Murray2015TopologyAD}
Devin Murray.
\newblock Topology and dynamics of the contracting boundary of cocompact
  {CAT}(0) spaces.
\newblock {\em Pacific Journal of Mathematics}, 299:89--116, 2015.

\bibitem[OOS09]{OOS_Lacunary_hyperbolic_groups}
Alexander~Yu. Ol'shanskii, Denis Osin, and Mark Sapir.
\newblock Lacunary hyperbolic groups.
\newblock {\em Geom. Topol.}, 13(4):2051--2140, 2009.
\newblock With an appendix by Michael Kapovich and Bruce Kleiner.

\bibitem[RST22]{Morse-local-to-global}
Jacob Russell, Davide Spriano, and Hung~Cong Tran.
\newblock The local-to-global property for {M}orse quasi-geodesics.
\newblock {\em Math. Z.}, 300(2):1557--1602, 2022.

\bibitem[RV21]{Rafi_Verberne_contracting}
Kasra Rafi and Yvon Verberne.
\newblock Geodesics in the mapping class group.
\newblock {\em Algebr. Geom. Topol.}, 21(6):2995--3017, 2021.

\bibitem[Sis16]{Sisto_Morse_in_aH}
Alessandro Sisto.
\newblock Quasi-convexity of hyperbolically embedded subgroups.
\newblock {\em Math. Z.}, 283(3-4):649--658, 2016.

\bibitem[Tra19]{Tran_SQC}
Hung~C. Tran.
\newblock On strongly quasiconvex subgroups.
\newblock {\em Geom. Topol.}, 23(3):1173--1235, 2019.

\end{thebibliography}
\bibliographystyle{alpha}

\end{document}